\documentclass[leqno]{amsart}
\usepackage{amsmath}
\usepackage{amsfonts}
\usepackage{amssymb}
\usepackage{pdfsync}
\usepackage{color,graphicx,tikz-cd}
\usepackage{a4wide}
\usepackage{amscd,amsthm,latexsym}
\usepackage{hyperref}
\usepackage{kotex}
\usepackage{relsize}
\usepackage{tikz}
\usepackage[misc]{ifsym}
\usetikzlibrary{snakes}
\usepackage{cite}

\allowdisplaybreaks %

\title{Product formulas for certain skew tableaux}

\author{Jang Soo Kim}
\address{
Department of Mathematics, Sungkyunkwan University, Suwon 16420,
South Korea}
\email{jangsookim@skku.edu}

\author{Meesue Yoo}
\address{
Applied Algebra and Optimization Research Center, Sungkyunkwan University, Suwon 16420,
South Korea}
\email{meesue.yoo@skku.edu (\Letter)}

\date{\today}

\thanks{The first author was supported by NRF grants \#2016R1D1A1A09917506 and \#2016R1A5A1008055.
The second author was supported by NRF grants \#2016R1A5A1008055 and \#2017R1C1B2005653.}



\date{\today}

\newtheorem{thm}{Theorem}[section]
\newtheorem{lem}[thm]{Lemma}

\newtheorem{cor}[thm]{Corollary}
\theoremstyle{definition}

\newtheorem{conj}[thm]{Conjecture}

\newcommand\flr[1]{\left\lfloor #1\right\rfloor}

\newcommand\LL{\mathcal{L}}

\newcommand\MM{\mathbf{M}}
\newcommand\VV{\mathbf{V}}

\newcommand\rdiag{\operatorname{rdiag}}
\newcommand\tr{\operatorname{tr}}
\newcommand\lm{{\lambda/\mu}}
\newcommand\nn{\mathbf{n}}
\newcommand\maj{\operatorname{maj}}
\newcommand\SSYT{\operatorname{SSYT}}
\newcommand\SYT{\operatorname{SYT}}
\newcommand\Par{\mathrm{Par}}
\newcommand\RPP{\operatorname{RPP}}
\newcommand\RST{\operatorname{RST}}

\newcommand\PPhi{\:\Phi}

\newcommand\cell[3]{
\def\i{#1} \def\j{#2} \def\entry{#3}
\draw (\j-1,-\i)--(\j,-\i)--(\j,-\i+1);
\node at (\j-.5,-\i+.5) {\entry};
}

\begin{document}

\begin{abstract}
The hook length formula gives a product formula for the number of standard Young tableaux of a partition shape. The number of standard Young tableaux of a skew shape does not always have a product formula. However, for some special skew shapes, there is a product formula. Recently, Morales, Pak and Panova joint with Krattenthaler conjectured a product formula for the number of standard Young tableaux of shape  $\lambda/\mu$ for $\lambda=((2a+c)^{c+a},(a+c)^a)$ and $\mu=(a+1,a^{a-1},1)$. They also conjectured a product formula for the number of standard Young tableaux of a certain skew shifted shape. In this paper we prove their conjectures using Selberg-type integrals. We also give a generalization of MacMahon's box theorem and a product formula for the trace generating function for a certain skew shape, which is a generalization of a recent result of Morales, Pak and Panova. 
\end{abstract}

\maketitle


\section{Introduction}

For a partition $\lambda$ of $n$, the number $f^\lambda$ of standard Young tableaux of shape $\lambda$ is given by the celebrated hook length formula
due to Frame, Robins and Thrall \cite{Frame1954}:
\[
f^\lambda = \frac{n!}{\prod_{(i,j)\in \lambda} h_\lambda(i,j)},
\]
where $h_\lambda(i,j)$ is the hook length $\lambda_i+\lambda'_j-i-j+1$. 
In general the number $f^{\lm}$ of standard Young tableaux of a skew shape does not have a product formula because it may have a large prime factor. However, in sporadic cases of skew shapes, some product formulas are known \cite{KimOh, KrattSch,MPP3}.

Recently, Naruse \cite{Naruse} found the following generalization of the hook length formula: 
\begin{equation}
  \label{eq:naruse}
f^{\lm} = |\lm|! \sum_{D\in\mathcal{E}(\lm)} \prod_{(i,j)\in \lambda\setminus D}\frac{1}{h_\lambda(i,j)},  
\end{equation}
where $\mathcal{E}(\lm)$ is the set of subsets $D\subset\lm$, called \emph{excited diagrams}, satisfying certain conditions. 
Morales, Pak and Panova \cite{MPP1} found the following $q$-analog of Naruse's hook length formula using semistandard Young tableaux:
\begin{equation}
  \label{eq:MPP1}
s_{\lm}(1,q,q^2,\dots) = \sum_{T\in\SSYT(\lm)}q^{|T|} =\sum_{D\in\mathcal{E}(\lm)} \prod_{(i,j)\in\lambda\setminus D}\frac{q^{\lambda_j'-i}}{1-q^{h_\lambda(i,j)}}.
\end{equation}
The precise definitions of notations used in the introduction  will be given in Section~\ref{sec:preliminaries}. 

Using \eqref{eq:MPP1}, Morales, Pak and Panova \cite{MPP3} found product formulas for the number of standard Young tableaux of certain skew shapes. 
In \cite{MPP3}, joint with Krattenthaler, they conjectured the following product formula.

\begin{conj}\cite[Conjecture~5.17]{MPP3}
\label{conj:MPP3-1}
Let $\lambda=((n+2a)^{n+a},(n+a)^a)$ and $\mu=(a+1,a^{a-1},1)$.
Then
\[
f^{\lm}= |\lm|!\frac{\PPhi(a)^4\PPhi(n)\PPhi(n+4a)}{\PPhi(2a)^2\PPhi(2n+4a)}
\cdot \frac{a^2((n^2+4an+2a^2)^2-a ^2)}{4a ^2-1},
\]
where $\PPhi(n) = \prod_{i=1}^{n-1}i!$.
\end{conj}

In Section~\ref{sec:con1}, we prove the following theorem, which is a generalization of Conjecture~\ref{conj:MPP3-1}.

\begin{thm}\label{thm:con1}
Let $\lambda = ((n+c+d)^{n+a},(n+c)^{d})$, $\mu=(c +1, c ^{a -1},1)$ and $\rho=\lm$.
Then
\begin{multline*}
f^\rho = |\rho|! \frac{\PPhi(n)\PPhi(a)\PPhi(b)\PPhi(c)\PPhi(d)\PPhi(n+a+b)\PPhi(n+c+d)\PPhi(n+a+b+c+d)}{\PPhi(a+b)\PPhi(c+d)\PPhi(n+a+c)\PPhi(n+b+d)\PPhi(2n+a+b+c+d)}\\
\times \frac{a b \big( n(n+a +b +c +d )(n(n+a+b+c+d)+(a+b)(c+d))+cd (a+b-1)(a +b +1) \big)}{(a +b -1)(a +b +1)}.
\end{multline*}
\end{thm}
See Figure~\ref{fig:rho} for the Young diagram of the skew shape $\rho$ used in Theorem~\ref{thm:con1}. Note that Conjecture~\ref{conj:MPP3-1} is obtained as a special case $a=b=c=d$ of Theorem~\ref{thm:con1}. Our proof of Theorem~\ref{thm:con1} consists of several steps. First, we consider the generating function for reverse plane partitions of shape $\rho$ and interpret it as a $q$-integral. Although it seems hopeless to evaluate the resulting $q$-integral because it has large irreducible factors, the  $q\to1$ limit becomes a Selberg-type integral which has a product formula. Using the well known connection between linear extensions and $P$-partitions of a poset, we obtain a product formula for $f^\rho$, which is then shown to be equivalent to the formula in Theorem~\ref{thm:con1}. 

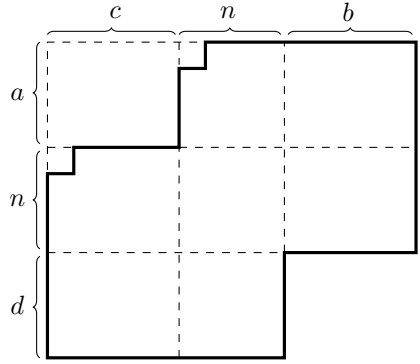
\begin{figure} 
\begin{tikzpicture}
\draw[very thick] (0,0)--(0,2.45)--(.35,2.45)--(.35,2.8)--(1.75,2.8)--(1.75,3.85)--(2.1,3.85)--(2.1,3.85)--(2.1,4.2)--(4.9,4.2)--(4.9,1.4)--(3.15,1.4)--(3.15,0)--(0,0)--cycle;
\draw[dashed] (0,1.4)--(3.15,1.4);
\draw[dashed] (1.75,2.8)--(4.9,2.8);
\draw[dashed] (1.75,0)--(1.75,2.8);
\draw[dashed] (3.15,1.4)--(3.15,4.2);

\draw[dashed] (0,2.5)--(0,4.2)--(2.1,4.2);
\draw[dashed] (0,2.8)--(0.35,2.8);
\draw[dashed] (1.75,4.2)--(1.75,3.9);
\draw[snake=brace] (0, 4.3)--(1.7, 4.3);
\draw[snake=brace] (1.75, 4.3)--(3.1, 4.3);
\draw[snake=brace] (3.2,4.3)--(4.85,4.3);

\draw[snake=brace] (-.1,2.85)--(-.1, 4.2);
\draw[snake=brace] (-.1,1.45)--(-.1, 2.75);
\draw[snake=brace] (-.1,0)--(-.1, 1.35);
\node (a) at (-.4, 2.1) {$n$};
\node (b) at (.9, 4.6) {$c$};
\node (c) at (-.4, 3.5) {$a$};
\node (d) at (4, 4.6) {$b$};
\node (d) at (2.4, 4.6) {$n$};
\node (e) at (-.4, .7) {$d$};
\end{tikzpicture}
\caption{The Young diagram of the skew shape $\rho$ in Theorem~\ref{thm:con1}.}
\label{fig:rho}
\end{figure}

For integers $n,a,b\ge0$ and $m\ge1$, let $\VV(n,a,b,m)$ denote the shifted skew shape $\lambda/\mu$ for 
\begin{equation}
  \label{eq:lambda}
\lambda = ((n+a+b, n+a+b-1,\dots, b+1)+(m-1) \delta_{n+a})^\ast,\qquad
\mu = (\delta_{a+1})^\ast,
\end{equation}
where $\nu^*$ denotes the shifted Young diagram of a strict partition $\nu$. 
See Figure~\ref{fig:pi} for the Young diagram of $\VV(n,a,b,m)$.
Morales, Pak and Panova \cite{MPP3}  also conjectured the following product formula.

\begin{figure}
\begin{tikzpicture}
\draw[very thick] (0, 1.5)--(0, 3)--(8.5,3)--(8.5,2.5)--(7.5,2.5)--(7.5,2)--(6.5,2)--(6.5,1.5)--(5.5,1.5)--(5.5,1)--(4.5,1)--(4.5,.5)--(3.5,.5)--(3.5,0)--(1.5,0)--(1.5,.5)--(1,.5)--(1,1)--(.5,1)--(.5,1.5)--(0,1.5)--cycle;
\draw (0, 2)--(2,2);
\draw (2,0)--(2,3);
\draw (3.5, 0)--(3.5, 3);
\draw[snake=brace] (0, 3.15)--(1.95,3.15);
\draw[snake=brace] (2.05, 3.15)--(3.5,3.15);
\draw[snake=brace] (-.15, 2.05)--(-.15, 2.95);
\draw[snake=brace] (-.15, 0.05)--(-.15, 1.95);
\node (q) at (-.45, 1) {$n$};
\node (r) at (-.45, 2.5) {$a$};
\node (n) at (1, 3.45) {$n$};
\node (s) at (2.8, 3.45) {$b$};
\node (d) at (5, 2.2) {$(m-1)\delta_{n+a}$};
\end{tikzpicture}
\caption{The Young diagram of $\VV(n,a,b,m)$.}\label{fig:pi}
\end{figure}
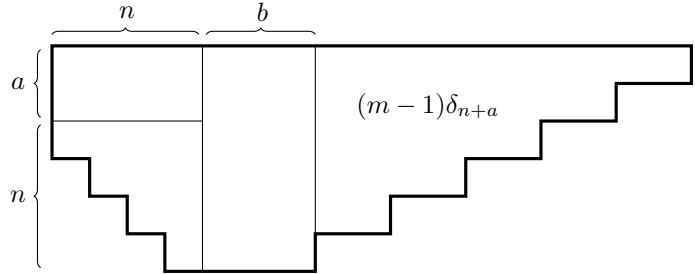

\begin{conj}\cite[Conjecture~9.6]{MPP3}
\label{conj:MPP3-2}
For $\pi=\VV(n,a,b,m)$, the number $g^\pi$ of standard Young tableaux of shape $\pi$ is 
\begin{equation*}
g^\pi = \frac{|\pi|!}{2^a} \cdot \frac{\PPhi(n+2a)\PPhi(a)}{\PPhi(2a)\PPhi(n+a)} \cdot \frac{\gimel(2a)\gimel(n)}{\gimel(n+2a)}
\prod_{(i,j)\in \lambda\backslash D}\frac{1}{h_{\lambda^*} (i,j)},
\end{equation*}
where $\gimel (n) = \prod_{i=1}^{\flr{n/2}} (n-2i)!$, $\lambda$ is given in \eqref{eq:lambda},
$D$ is the set of cells $(i,n+j)$ with $1\le i\le j\le n$ and $h_{\lambda^*} (i,j)$ is the shifted hook length.
\end{conj}
In Section~\ref{sec:conj2} we prove Conjecture~\ref{conj:MPP3-2}  by a similar approach used in the proof of Theorem~\ref{thm:con1}. 

For integers $n, a,b,c,d\ge0$ and $m\ge1$, we let $\MM(n, a,b,c,d, m)$ denote the skew shape $\lambda/(c^a)$, where  $\lambda=((n+c+b)^{n+a})+(m-1)\delta_{n+a})\cup \nu '$ and $\nu=(d^{n+c})+(m-1)\delta_{n+c}$, see Figure~\ref{fig:biglambda}. 

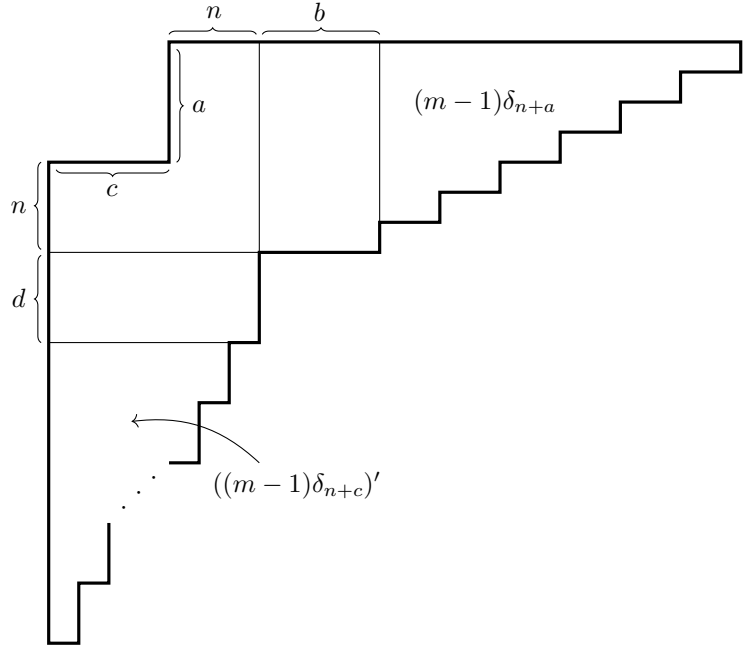
\begin{figure}
\begin{tikzpicture}
\draw[very thick] (.8,1.6)--(.8,.8)--(.4,.8)--(.4,0)--(0,0)--(0,6.4)--(1.6,6.4)--(1.6,8)--(9.2,8)--(9.2,7.6)--(8.4,7.6)--(8.4,7.2)--(7.6,7.2)--(7.6,6.8)--(6.8,6.8)--(6.8,6.4)--(6,6.4)--(6,6)--(5.2,6)--(5.2,5.6)--(4.4,5.6)--(4.4,5.2)--(3.2,5.2)--(2.8,5.2)--(2.8,4)--(2.4,4)--(2.4,3.2)--(2,3.2)--(2,2.4)--(1.6,2.4);
\draw (0,4)--(2.8,4);
\draw (0,5.2)--(3.2,5.2);
\draw (2.8,5.2)--(2.8,8);
\draw (4.4,5.2)--(4.4,8);
\draw[snake=brace] (1.6,8.1)--(2.76,8.1);
\draw[snake=brace] (2.84, 8.1)--(4.4,8.1);
\draw[snake=brace] (1.7,7.9)--(1.7,6.4);
\draw[snake=brace] (1.6, 6.3)--(.1,6.3);
\draw[snake=brace] (-.1,5.24)--(-.1, 6.4);
\draw[snake=brace] (-.1,4)--(-.1, 5.16);
\node (a) at (2, 7.15) {$a$};
\node (c) at (.84, 6.02) {$c$};
\node (n1) at (2.2,8.4) {$n$};
\node (b) at (3.6,8.4) {$b$};
\node (n2) at (-.4,5.8) {$n$};
\node (d) at (-.4,4.6) {$d$};
\node (r) at (5.8, 7.2) {$(m-1)\delta_{n+a}$};
\node (d1) at (1,1.8) {$\cdot$};
\node (d2) at (1.2,2) {$\cdot$};
\node (d3) at (1.4,2.2) {$\cdot$};
\draw [->] (2.8, 2.4) to [bend right=25] (1.1, 2.95);
\node (l) at (3.3, 2.1) {$((m-1)\delta_{n+c})'$};
\end{tikzpicture}
\caption{The Young diagram of the skew shape $\MM(n,a,b,c,d, m)$.}\label{fig:biglambda}
\end{figure}

In Section~\ref{sec:GMBT} we show that
for $\pi=\MM(n,a,b,c,d,1)$, the generating function
$s_\pi(1,q,q^2,\dots,q^N)$ for SSYTs of shape $\pi$ with bounded entries also has a product formula. 

\begin{thm}\label{thm:GMBT1}
For $\pi=\MM(n,a,b,c,d,1)$ and an integer $N\ge0$, we have 
\begin{multline*}
  s_\pi(1,q,q^2,\dots,q^N) =q^{\sum_{(i,j)\in \lambda/(c^a)}(\lambda_j ' -i)}
 \prod_{i=1}^{b}(q^{N-a +1+i};q)_{a} \prod_{i=1}^{d} (q^{N+2-i};q)_{c}\\
\times \prod_{i=1}^n (q^{N-n-a-d+1+i};q)_{n+a+b+c+d} \prod_{i=1}^n \prod_{j=1}^{a}\prod_{k=1}^{c}\frac{1-q^{i+j+k-1}}{1-q^{i+j+k-2}}
\cdot \prod_{(i,j)\in \lambda\setminus(0^n,c^a)}\frac{1}{1-q^{h_\lambda (i,j)}},
\end{multline*}
where $(0^n,c^a)$ is the set of cells in the $a\times c$ rectangle, starting from the $(n+1)$-st row. (In other words, $(0^n,c^a) = (c^{n+a})/(c^n)$.)
\end{thm}

Recall that MacMahon's box theorem states that
\begin{equation}
  \label{eq:MBT}
\sum_{\substack{T\in\RPP(b^a)\\\max(T)\le c}} q^{|T|} = 
\prod_{i=1}^a \prod_{j=1}^{b}\prod_{k=1}^{c}\frac{1-q^{i+j+k-1}}{1-q^{i+j+k-2}}.
\end{equation}
By the simple connection between SSYTs and RPPs of any partition shape,
\eqref{eq:MBT} is equivalent to 
\begin{equation}
  \label{eq:MBT2}
\sum_{\substack{T\in\SSYT(b^a)\\\max(T)\le c+a-1}} q^{|T|} = q^{b\binom a2}
\prod_{i=1}^a \prod_{j=1}^{b}\prod_{k=1}^{c}\frac{1-q^{i+j+k-1}}{1-q^{i+j+k-2}}.
\end{equation}
Since the shape $\MM(n,a,b,c,d,1)$ is more general than $(b^a)$, Theorem~\ref{thm:GMBT1} can be considered as a generalization of MacMahon's box theorem. 

Using \eqref{eq:MPP1}, Morales, Pak and Panova \cite{MPP3} found a product formula for the generating function for semistandard Young tableaux of shape $\MM(n, a,b,c,d, m)$. 

\begin{thm} \cite[Theorem~4.2]{MPP3}
Let $\lm=\MM(n, a,b,c,d, m)$. Then
\label{thm:MPP3}
\[
\sum_{T\in\SSYT(\lm)}q^{|T|}  = q^{\sum_{(i,j)\in \lambda/(c^{a})} (\lambda_j '-i)}
 \prod_{i=1}^n \prod_{j=1}^{a}\prod_{k=1}^{c}\frac{1-q^{m(i+j+k-1)}}{1-q^{m(i+j+k-2)}}
\prod_{(i,j)\in \lambda\setminus(0^n, c^{a})}\frac{1}{1-q^{h_\lambda(i,j)}}.
\]
\end{thm}

In Section~\ref{sec:trace} we show the following trace generating function formula, which is a generalization of Theorem~\ref{thm:MPP3}. 

\begin{thm}\label{thm:trace}
Let $\pi=\MM(n,a,b,c,d,m)$. Then
\begin{multline*}
\sum_{T\in\SSYT(\pi)} x^{\tr(T)}q^{|T|} =
x^{n a +\binom{n}{2}}q^{\sum_{(i,j)\in \lambda/(c^a)} (\lambda_j '-i)}\\
\times  \prod_{i=1}^n \prod_{j=1}^{a}\prod_{k=1}^{c}\frac{1-q^{m(i+j+k-1)}}{1-q^{m(i+j+k-2)}}
\prod_{(i,j)\in \lambda\setminus (0^n, c^a)}\frac{1}{1-x^{\chi(i,j)}q^{h_\lambda (i,j)}},
\end{multline*}
where 
$$
\chi(i,j)=\begin{cases} 1, &\text{ if } (i,j) \in (n+c)^{n+a},\\
0,& \text{ otherwise}.\end{cases}
$$
\end{thm}

\section{Preliminaries}\label{sec:preliminaries}

The following notations will be used throughout this paper:
\[
(2n-1)!! = 1\cdot 3 \cdots (2n-1),\qquad (a;q)_n = (1-a)(1-aq)\cdots (1-aq^{n-1}),
\]
\[
\PPhi(n) = \prod_{i=1}^{n-1} i !,\qquad
\gimel (n) = \prod_{i=1}^{\flr{n/2}} (n-2i)!, 
\]
\[
\PPhi_q (n) = \prod_{i=1}^{n-1}(q;q)_i,\qquad
\gimel _q (n) = \prod_{i=1}^{\flr{n/2}}(q;q)_{n-2i}.
\]

A \emph{partition} is a sequence $\lambda=(\lambda_1,\dots,\lambda_k)$ of integers $\lambda_1\ge\lambda_2\ge\dots\ge\lambda_k\ge 0$. 
Each $\lambda_i>0$ is called a \emph{part} of $\lambda$. The \emph{length} $\ell(\lambda)$ of $\lambda$ is the number of parts in $\lambda$. 
We denote by $\Par_n$ the set of partitions with at most $n$ parts. We will use the convention that  $\lambda_i=0$ for all $i>\ell(\lambda)$. We define 
\[
\delta_n =  (n-1,n-2,\dots,1,0)\in\Par_n.
\]

For a partition $\lambda$, let
\[
\nn(\lambda)=\sum_{i=1}^{\ell(\lambda)} (i-1)\lambda_i.
\]
For $\lambda\in\Par_n$ and a sequence $x=(x_1,\dots,x_n)$ of variables, we define
\begin{align*}
{a}_\lambda(x) &= \det(x_{j}^{\lambda_i+n-i})_{i,j=1}^n,\\
\overline{a}_\lambda(x) &= \det(x_{n+1-j}^{\lambda_i+n-i})_{i,j=1}^n = (-1)^{\binom n2}{a}_\lambda(x),\\
\Delta(x) &= a_{\delta_n}(x)= \prod_{1\le i<j\le n} (x_i-x_j),\\
\overline{\Delta}(x) &= \overline{a}_{\delta_n}(x)= \prod_{1\le i<j\le n} (x_j-x_i).
\end{align*}

For a partition $\lambda$, if $m_i$ is the number of parts equal to $i$ in $\lambda$ for $i\ge1$, we also write $\lambda$ as $(M^{m_M},\dots,2^{m_2},1^{m_1})$,
where $M$ is an integer greater than or equal to the largest part of $\lambda$. 
For two partitions $\lambda,\mu\in\Par_n$, we define $\lambda+\mu$ to be the partition $\nu\in\Par_n$ given by $\nu_i=\lambda_i+\mu_i$. 
We also define $\lambda\cup \mu=(M^{c_M},\dots,2^{c_2},1^{c_1})$, where $\lambda=(M^{a_M},\dots,2^{a_2},1^{a_1})$,
$\mu=(M^{b_M},\dots,2^{b_2},1^{b_1})$ and $c_i=a_i+b_i$ for $i\ge1$. 
For a partition $\lambda=(\lambda_1,\dots,\lambda_k)$ and an integer $m\ge0$, we define
$m\lambda=(m\lambda_1,\dots,m\lambda_k)$.

We will identify a partition $\lambda$ with its \emph{Young diagram}  $\{(i,j)\in \mathbb{Z}\times\mathbb{Z}: 1\le i\le \ell(\lambda), 1\le j\le \lambda_i\}$.
The \emph{transpose} $\lambda'$ of $\lambda$ is the partition whose Young diagram is given by
$\{(j,i):(i,j)\in\lambda\}$. More generally, for any subset $D\subseteq\mathbb{Z}\times\mathbb{Z}$, we define
$D'=\{(j,i)\in \mathbb{Z}\times\mathbb{Z}: (i,j)\in D\}$.

For two partitions $\lambda$ and $\mu$, the notation $\mu\subseteq\lambda$ means that the Young diagram of $\mu$ is a subset of the Young diagram of $\lambda$. In this case, the \emph{skew shape} $\lm$ is defined to be the set theoretic difference $\lambda\setminus\mu$. 
If $\lambda$ is a \emph{strict partition}, i.e., $\lambda_1>\dots>\lambda_{\ell(\lambda)}$, the \emph{shifted Young diagram},
denoted by $\lambda^*$, is the set $\{(i,j)\in \mathbb{Z}\times\mathbb{Z}: 1\le i\le \ell(\lambda), ~i\le j\le \lambda_i+i-1\}$.
For two strict partitions $\lambda$ and $\mu$ with $\mu^*\subseteq\lambda^*$, the \emph{shifted skew shape} 
$\lambda^*/\mu^*$ is also defined to be the set theoretic difference $\lambda^*\setminus\mu^*$. 
The Young diagram (resp.~shifted Young diagram) of a partition $\lambda$ will also be considered as the skew shape $\lambda/\emptyset$ 
(resp.~shifted skew shape $\lambda^*/\emptyset^*$).

Let $\pi$ be a skew shape $\lm$ or a shifted skew shape $\lambda^*/\mu^*$. Then $\pi$ is represented by an array of \emph{cells} as shown in Figure~\ref{fig:YD}. We can identify each element in $\pi$ with the corresponding cell in the graphical representation of $\pi$. The \emph{size} of $\pi$, denoted by $|\pi|$, is the number of cells in $\pi$.  

\begin{figure}
  \centering
  \begin{tikzpicture}[scale=.5]
\cell13{} \cell14{}
\cell22{} \cell23{}
\cell31{}
\draw (0,-3) -- (0,-2) -- (1,-2) -- (1,-1) -- (2,-1)--(2,0)--(4,0);
\draw[dotted] (0,-2)--(0,0)--(2,0);
  \end{tikzpicture} \qquad \qquad
  \begin{tikzpicture}[scale=.5]
\cell14{} \cell15{} \cell16{}
\cell23{} \cell24{} \cell25{}\cell26{}
\cell33{} \cell34{}
\draw (2,-3) -- (2,-1) -- (3,-1) -- (3,0) -- (6,0);
\draw[dotted] (2,-2)--(1,-2)--(1,-1)--(0,-1)--(0,0)--(3,0);
  \end{tikzpicture}
  \caption{The skew shape $(4,3,1)/(2,1)$ on the left and the shifted skew shape $(6,5,2)^*/(3,1)^*$ on the right.}
  \label{fig:YD}
\end{figure}
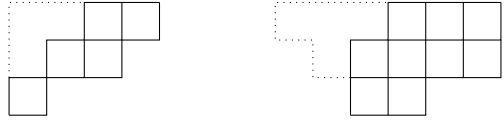
 
A \emph{standard Young tableau (SYT)} of shape $\pi$ is a filling of the cells in $\pi$ with  $1,2,\dots,|\pi|$ such that the entries are increasing in each row and in each column. A \emph{semistandard Young tableau (SSYT)} of shape $\pi$ is a filling of the cells in $\pi$ with nonnegative integers such that the entries are weakly increasing in each row and strictly increasing in each column.  A \emph{row strict tableau (RST)} of shape $\pi$ is a filling of the cells in $\pi$ with nonnegative integers such that the entries are strictly increasing in each row and weakly increasing in each column. A \emph{reverse plane partition (RPP)} of shape $\pi$ is a filling of the cells in $\pi$ with nonnegative integers such that the entries which are weakly increasing in each row and in each column. See Figure~\ref{fig:tab} for examples of these objects. 
We denote by $\SYT(\pi)$, $\SSYT(\pi)$, $\RST(\pi)$ and $\RPP(\pi)$, respectively,  the set of 
SYTs, SSYTs, RSTs and RPPs of shape $\pi$. We will simply call an element in one of these sets a \emph{tableau} of shape $\pi$.
For a tableau $T$, we denote by $|T|$ the sum of the entries in $T$, by $\max(T)$ the largest entry in $T$ and by $\min(T)$ the smallest entry in $T$.

\begin{figure}
  \centering
  \begin{tikzpicture}[scale=.5]
\cell11{1} \cell12{3} \cell13{4} \cell14{8}
\cell21{2} \cell22{5} \cell23{7}
\cell31{6}
\draw (0,-3) -- (0,0)--(4,0);
  \end{tikzpicture} \qquad 
  \begin{tikzpicture}[scale=.5]
\cell11{0} \cell12{0} \cell13{1} \cell14{1}
\cell21{1} \cell22{2} \cell23{4}
\cell31{5}
\draw (0,-3) -- (0,0)--(4,0);
  \end{tikzpicture} \qquad 
  \begin{tikzpicture}[scale=.5]
\cell11{0} \cell12{1} \cell13{3} \cell14{7}
\cell21{0} \cell22{1} \cell23{4}
\cell31{1}
\draw (0,-3) -- (0,0)--(4,0);
  \end{tikzpicture} \qquad 
  \begin{tikzpicture}[scale=.5]
\cell11{0} \cell12{0} \cell13{1} \cell14{4}
\cell21{0} \cell22{1} \cell23{1}
\cell31{3}
\draw (0,-3) -- (0,0)--(4,0);
  \end{tikzpicture} 
  \caption{An SYT, an SSYT, an RST and an RPP of shape $(4,3,1)$ from left to right.}
  \label{fig:tab}
\end{figure}
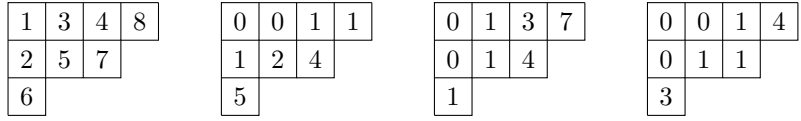

The following well known lemma tells us that the number of standard Young tableaux can be computed from the generating function of RPPs or SSYTs. 

\begin{lem}\label{lem:SYT_RPP}
For any skew shape or shifted skew shape $\pi$, we have
\[
f^\pi=\lim_{q\to1} (q;q)_{|\pi|}\sum_{T\in \RPP (\pi)}q^{|T|}=\lim_{q\to1} (q;q)_{|\pi|}\sum_{T\in \SSYT(\pi)}q^{|T|}.  
\]  
\end{lem}
\begin{proof}
This can be proved using the $(P,\omega)$-partition theory. 
We follow the terminologies in \cite[Section~3.15]{EC1}.
Let $P$ be the poset whose elements are the cells $(i,j)\in\pi$ with relation
$(i,j)\le_P(i',j')$ if   $i\ge i'$ and $j\ge j'$. Let $\omega$ be any natural labeling of $P$.
Then $\SYT(\pi)$ is in bijection with the set $\LL(P,\omega)$ of linear extensions of $P$, 
and the RPPs of shape $\pi$ can be considered as the $(P,\omega)$-partitions. By \cite[Theorem~3.15.7]{EC1},
\[
\sum_{w\in \LL(P,\omega)} q^{\maj (w)} = (q;q)_{|\pi|}\sum_{T\in \RPP (\pi)}q^{|T|}.  
\]
Since $f^\pi=|\SYT(\pi)|=|\LL(P,\omega)|$, by taking the $q\to1$ limit, we obtain the first identity. 
The second identity can be proved similarly, see \cite[Proposition~7.19.11]{EC2}.
\end{proof}

Let $T$ be a tableau of shifted shape $\lambda^*$ for a strict partition $\lambda$ with $\ell(\lambda)=\ell$. The \emph{reverse diagonal} of $T$ is the partition $\rdiag(T)=(d_{\ell}, d_{\ell-1},\dots,d_1)$, where
$d_i$ is the entry in the cell $(i,i)$ of $T$ for $1\le i\le \ell$. For example, if $T_1,T_2$ and $T_3$ are the tableaux in
Figure~\ref{fig:rdiag} from left to right, then $\rdiag(T_1)=(5,1,0)$, $\rdiag(T_2)=(4,1,0)$ and $\rdiag(T_3)=(2,0,0)$.

\begin{figure}
  \centering
  \begin{tikzpicture}[scale=.5]
\cell11{0} \cell12{0} \cell13{0} \cell14{1} \cell15{2}
\cell22{1} \cell23{2} \cell24{2}
\cell33{5} \cell34{6}
\draw (2,-3) -- (2,-2); \draw (1,-2) -- (1,-1); \draw (0,-1)--(0,0)--(5,0);
  \end{tikzpicture} \qquad 
  \begin{tikzpicture}[scale=.5]
\cell11{0} \cell12{1} \cell13{2} \cell14{4} \cell15{6}
\cell22{1} \cell23{3} \cell24{4}
\cell33{4} \cell34{6}
\draw (2,-3) -- (2,-2); \draw (1,-2) -- (1,-1); \draw (0,-1)--(0,0)--(5,0);
  \end{tikzpicture} \qquad 
  \begin{tikzpicture}[scale=.5]
\cell11{0} \cell12{0} \cell13{0} \cell14{1} \cell15{2}
\cell22{0} \cell23{1} \cell24{1}
\cell33{2} \cell34{3}
\draw (2,-3) -- (2,-2); \draw (1,-2) -- (1,-1); \draw (0,-1)--(0,0)--(5,0);
  \end{tikzpicture} 
  \caption{An SSYT, an RST and an RPP of shape $(5,3,2)^*$ from left to right.}
  \label{fig:rdiag}
\end{figure}

For a partition $\mu=(\mu_1,\dots, \mu_n)$, we denote $q^\mu=(q^{\mu_1},\dots, q^{\mu_n})$. 
The following theorem is a key ingredient in this paper. 

\begin{thm}\cite[Theorem 8.7]{KimStanton17}\label{thm:KS_8.7}
For $\lambda, \mu\in \Par_n$, we have 
\[
\sum_{\substack{T\in \RPP((\delta_{n+1}+\lambda)^\ast) \\ \rdiag (T)=\mu}} q^{|T|} = 
\frac{q^{-\nn(\delta_{n+1}+\lambda)}}{\prod_{j=1}^n (q;q)_{\lambda_j +n -j}}q^{|\mu+\delta_n|} \overline{a}_{\lambda+\delta_n }(q^{\mu +\delta_n}).
\]
\end{thm}

Given an RPP, by adding $(i-1)$'s to all the cells in the $i$-th column, we can get an RST, and 
similarly, by adding $(i-1)$'s to the cells in the $i$-th row, we can get an SSYT. For example, the SSYT and the RST in Figure~\ref{fig:rdiag} 
are obtained in this way from the RPP on the right. Note that
if $T\in\RPP((\delta_{n+1}+\lambda)^*)$ and $\rdiag(T)=\mu$, then the resulting SSYT or RST $T'$ has $\rdiag(T') = \mu+\delta_n$.
Applying this process to Theorem \ref{thm:KS_8.7}, we get the generating functions for SSYT and RST with fixed diagonal. 

\begin{cor}\label{cor:gfs}
For $\lambda, \nu\in \Par_n$, we have 
\begin{align}
\sum_{\substack{T\in \SSYT((\delta_{n+1}+\lambda)^\ast) \\ \rdiag (T)=\nu}} q^{|T|} 
&= \frac{q^{|\nu|}}{\prod_{j=1}^n (q;q)_{\lambda_j +n -j}} \overline{a}_{\lambda+\delta_n }(q^{\nu}),\label{eqn:ssytgf1}\\
\sum_{\substack{T\in \RST((\delta_{n+1}+\lambda)^\ast) \\ \rdiag (T)=\nu}} q^{|T|} 
&= \frac{q^{|\nu|+\nn(\lambda')-\nn(\lambda)+n|\lambda|+\binom{n+1}{3}}}{\prod_{j=1}^n (q;q)_{\lambda_j +n -j}} \overline{a}_{\lambda+\delta_n }(q^{\nu}).\label{eqn:rstgf}
\end{align}
\end{cor}

For a partition $\lambda$, the \emph{hook length} of $(i,j)\in \lambda$ is defined by
\[
h_\lambda(i,j)=\lambda_i+\lambda'_j-i-j+1.
\]
Now let $\lambda$ be a strict partition with $\ell(\lambda)=n$. Then
$\lambda=\delta_{n+1}+\mu$ for some $\mu\in\Par_n$. The \emph{shifted hook length} of $(i,j)\in\lambda^*$ is defined by
\[
h_{\lambda^*}(i,j)=
\begin{cases}
n+1+\mu_i, &\mbox{if  $i=j$,}\\
\mu_i+\mu_{j}+2(n+1)-i-j, & \mbox{if  $i<j\le n$,}\\  
h_\mu(i,j-n), & \mbox{if  $j>n$.}
\end{cases}
\]
For an example, see Figure~\ref{fig:slhex}. 
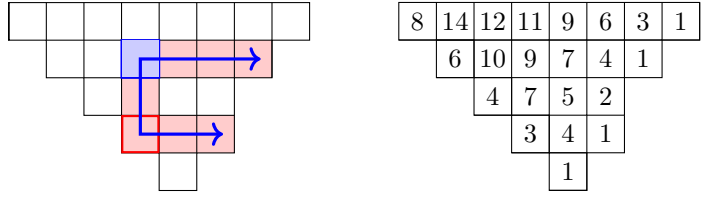
\begin{figure}
  \centering
\begin{tikzpicture}
\fill [fill=red!20!white] (2,1.5) rectangle (3.5,2);
\fill [fill=red!20!white] (1.5,.5) rectangle (2,1.5);
\fill [fill=red!20!white] (2,.5) rectangle (3,1);
\draw (0,2)--(0,2.5)--(4,2.5)--(4,2)--(3.5,2)--(3.5,1.5)--(3,1.5)--(3,0.5)--(2.5,0.5)--(2.5,0)--(2,0)--(2,.5)--(1.5,.5)--(1.5,1)--(1,1)--(1,1.5)--(.5,1.5)--(.5,2)--(0,2)--(0,2.5)--cycle;
\draw (.5,2)--(3.5,2);
\draw (1,1.5)--(3,1.5);
\draw (1.5,1)--(3,1);
\draw (2, .5)--(2.5,.5);
\draw (.5, 2)--(.5, 2.5);
\draw (1, 1.5)--(1,2.5);
\draw (1.5, 1)--(1.5, 2.5);
\draw (2, .5)--(2, 2.5);
\draw (2.5, .5)--(2.5, 2.5);
\draw (3, 1.5)--(3, 2.5);
\draw (3.5, 1.5)--(3.5, 2.5);
\draw [blue, thick] (1.5, 1.5) rectangle (2,2);
\fill [fill=blue!20!white] (1.51, 1.51) rectangle (1.99,1.99);
\draw [red, thick] (1.51, .51) rectangle (1.99,0.99);
\draw [blue,very thick, <->] (3.35, 1.75)--(1.75,1.75)--(1.75,.75)--(2.85,.75);
\end{tikzpicture}
\qquad\quad
\begin{tikzpicture}[scale=.5]
\draw (0,-1)--(0,0)--(8,0)--(8,-1)--(0,-1)--cycle;
\draw (1, -2)--(1,-1)--(2,-1)--(2,-2)--(1,-2)--cycle;
\draw (2, -3)--(2,-2)--(3,-2)--(3,-3)--(2,-3)--cycle;
\draw (3, -4)--(3,-3)--(4,-3)--(4,-4)--(3,-4)--cycle;
\draw (4, -5)--(4,-4)--(5,-4)--(5,-5)--(4,-5)--cycle;
\cell11{8} \cell12{14}\cell13{12}\cell14{11}\cell15{9}\cell16{6}\cell17{3}\cell18{1}
\cell22{6}\cell23{10}\cell24{9}\cell25{7}\cell26{4}\cell27{1}
\cell33{4}\cell34{7}\cell35{5}\cell36{2}
\cell44{3}\cell45{4}\cell46{1}
\cell55{1}
\end{tikzpicture}
\caption{An example of $\lambda=\delta_6 +(3,2,1,1)$ for $n=5$. The left depicts the set of cells contained in the hook of the cell $(2,4)$ (blue cell in the picture) and the right shows 
the shifted hook lengths of $\lambda^\ast$. Note that the red cell in the diagonal is counted twice.  }\label{fig:slhex}
\end{figure}

The \emph{$q$-integral} of $f(x)$ over $[a,b]$ is defined by
\[
\int_{a}^b f(x)d_qx=(1-q)\sum_{i\ge0} (f(bq^i)bq^i-f(aq^i)aq^i),
\]
where $0<q<1$ and the sum is assumed to absolutely converge.
We also define the multivariate $q$-integral 
\[
\int_{0\le x_1\le \cdots \le x_n\le 1}f(x_1,\dots, x_n)d_q x_1\cdots d_q x_n
=\int_0^1\int_0^{x_n}\int_0^{x_{n-1}}\dots\int_0^{x_2}f(x_1,\dots, x_n)d_q x_1\cdots d_q x_n.
\]
The following lemma tells us how we make a change of variables in the $q$-integral. 
 
\begin{lem}\label{lem:cov}
Let $f(x,q)$ be a function with two variables $x$ and $q$. For an integer $m\ge0$, we have
$$\int_0 ^1 f(x^m, q^m)d_q x =\frac{1-q}{1-p} \int_0 ^1 f(x, p) \cdot x^{\frac{1-m}{m}}d_p x,$$
where $p=q^m$. 
\end{lem}
\begin{proof}
By the definition of the $q$-integral, the left hand side is 
\begin{align*}
\int_0 ^1 f(x^m, q^m)d_q x &= (1-q)\sum_{i\ge 0} f(q^{mi},q^m) q^i\\
&= (1-q)\sum_{i\ge 0}f(p^i, p)p^{i/m}\\
&= \frac{1-q}{1-p}\cdot  (1-p)\sum_{i\ge 0}f(p^i, p)p^{\frac{1-m}{m}i}\cdot p^i,
\end{align*}
which is equal to the right hand side. 
\end{proof}

 The following lemma explains how we write the $q$-summation of certain functions in terms of the $q$-integral.

\begin{lem}\cite[Lemma 4.3]{KimStanton17}\label{lem:gf}
For a function $f(x_1,\dots, x_n)$ satisfying $f(x_1,\dots, x_n)=0$ if $ x_i =x_j$ for any $i\ne j$, 
$$\sum_{\mu\in \Par_n}q^{|\mu+\delta_n|}f(q^{\mu+\delta_n})=\frac{1}{(1-q)^n}\int_{0\le x_1\le \cdots \le x_n\le 1}f(x_1,\dots, x_n)d_q x_1\cdots d_q x_n.$$
\end{lem}
Combining \eqref{eqn:ssytgf1} and \eqref{eqn:rstgf} with Lemma \ref{lem:gf} gives a formula for the generating function of semistandard Young tableaux or row strict tableaux of shifted shapes.
This method is used throughout the paper to derive the generating functions of semistandard Young tableaux of certain skew shapes. 

We finish this section by showing that MacMahon's box theorem \eqref{eq:MBT} can be obtained from Theorem~\ref{thm:KS_8.7}. 
The idea is to consider the reverse plane partitions of a shifted shape $(\delta_n)^*$ with a suitable choice of the reverse diagonal entries so that some entries are forced to be fixed. This technique will be used several times in this paper. Now we restate 
MacMahon's box theorem and give a proof using this technique.

\begin{thm}
We have
\[
\sum_{\substack{T\in\RPP(b^a)\\\max(T)\le c}} q^{|T|} = 
\prod_{i=1}^a \prod_{j=1}^{b}\prod_{k=1}^{c}\frac{1-q^{i+j+k-1}}{1-q^{i+j+k-2}}.
\]
\end{thm}
\begin{proof}
Let $R\in \RPP((\delta_{a+b+1})^*)$ with $\rdiag(R)=(c^b,0^a)$. Observe that the entries in the triangular region $\{(i,j):1\le i\le j\le a\}$ are all equal to $0$
and  the entries in the triangular region $\{(i,j):a+1\le i\le j\le a+b\}$ are all equal to $c$. Let $T$ be the reverse plane partition obtained from $R$ by removing these fixed entries. Then the map $R\mapsto T$ gives a bijection
from $\{R\in\RPP((\delta_{a+b+1})^*):\rdiag(R)=(c^b,0^a)\}$ to 
$\{T\in\RPP(b^a):\max(T)\le c\}$.
Thus, by Theorem~\ref{thm:KS_8.7},
\begin{align*}
\sum_{\substack{T\in\RPP(b^a)\\\max(T)\le c}} q^{|T|} &= 
q^{-c\binom{b+1}2}\sum_{\substack{R\in\RPP((\delta_{a+b+1})^*)\\\rdiag(R)=(c^b,0^a) }} q^{|R|}\\
& = q^{-c\binom{b+1}2}
\frac{q^{-\nn(\delta_{a+b+1})}}{\prod_{j=1}^{a+b} (q;q)_{a+b -j}}q^{|(c^b,0^a)+\delta_{a+b+1}|} \overline{a}_{\delta_{a+b+1}}(q^{(c^b,0^a)+\delta_{a+b+1}}).
\end{align*}
Since 
\[
\overline{a}_{\delta_{a+b+1}}(x_1,\dots,x_{a+b+1})=\overline{\Delta}(x_1,\dots,x_{a+b+1}) = \prod_{1\le i<j\le a+b+1}(x_j-x_i),
\]
we can simplify the above formula to obtain the theorem.
\end{proof}

\section{Number of standard Young tableaux of certain skew shape}
\label{sec:con1}

In this section we prove Theorem~\ref{thm:con1}, which is restated below.

\begin{figure}

\begin{tikzpicture}
\draw[very thick] (0,0)--(0,2.45)--(.35,2.45)--(.35,2.8)--(1.75,2.8)--(1.75,3.85)--(2.1,3.85)--(2.1,3.85)--(2.1,4.2)--(4.9,4.2)--(4.9,1.4)--(3.15,1.4)--(3.15,0)--(0,0)--cycle;
\draw[dashed] (0,1.4)--(3.15,1.4);
\draw[dashed] (1.75,2.8)--(4.9,2.8);
\draw[dashed] (1.75,0)--(1.75,2.8);
\draw[dashed] (3.15,1.4)--(3.15,4.2);

\draw[dashed] (0,2.5)--(0,4.2)--(2.1,4.2);
\draw[dashed] (0,2.8)--(0.35,2.8);
\draw[dashed] (1.75,4.2)--(1.75,3.9);
\draw[snake=brace] (0, 4.3)--(1.7, 4.3);
\draw[snake=brace] (1.75, 4.3)--(3.1, 4.3);
\draw[snake=brace] (3.2,4.3)--(4.85,4.3);

\draw[snake=brace] (-.1,2.85)--(-.1, 4.2);
\draw[snake=brace] (-.1,1.45)--(-.1, 2.75);
\draw[snake=brace] (-.1,0)--(-.1, 1.35);
\node (a) at (-.4, 2.1) {$n$};
\node (b) at (.9, 4.6) {$c$};
\node (c) at (-.4, 3.5) {$a$};
\node (d) at (4, 4.6) {$b$};
\node (d) at (2.4, 4.6) {$n$};
\node (e) at (-.4, .7) {$d$};
\end{tikzpicture}
\qquad
\begin{tikzpicture}

\draw[dashed] (0,1.4)--(3.15,1.4);
\draw[dashed] (0,2.8)--(4.9,2.8);
\draw[dashed] (1.75,0)--(1.75,2.8);
\draw[dashed] (3.15,1.4)--(3.15,4.2);
\draw[dashed] (0,0)--(0,4.2)--(2.1,4.2);
\draw[dashed] (1.75,4.2)--(1.75,3.9);

\draw[very thick, blue] (1.75,2.8)--(1.75,3.85)--(1.75,4.2)--(4.9,4.2)--(4.9,1.4)--(3.15,1.4)--(2.8,1.4)--(2.8,1.75)--(2.45,1.75)--(2.45,2.1)--(2.1,2.1)--(2.1,2.45)--(1.75,2.45)--cycle;
\draw[very thick,red] (0,0)--(3.15,0)--(3.15,1.75)--(2.8,1.75)--(2.8,2.1)--(2.45,2.1)--(2.45,2.45)--(2.1,2.45)--(2.1,2.8)--(0,2.8)--cycle;

\draw[snake=brace] (0, 4.3)--(1.7, 4.3);
\draw[snake=brace] (1.75, 4.3)--(3.1, 4.3);
\draw[snake=brace] (3.2,4.3)--(4.85,4.3);

\draw[snake=brace] (-.1,2.85)--(-.1, 4.2);
\draw[snake=brace] (-.1,1.45)--(-.1, 2.75);
\draw[snake=brace] (-.1,0)--(-.1, 1.35);
\node (a) at (-.4, 2.1) {$n$};
\node (b) at (.9, 4.6) {$c$};
\node (c) at (-.4, 3.5) {$a$};
\node (d) at (4, 4.6) {$b$};
\node (d) at (2.4, 4.6) {$n$};
\node (e) at (-.4, .7) {$d$};
\end{tikzpicture}
\caption{The skew shape $\rho$ is the  diagram on the left. The skew shape $\rho^r$ is the blue diagram on the right.
The skew shape $\rho^l$ is the red diagram on the right.}\label{fig:decomp}
\end{figure}
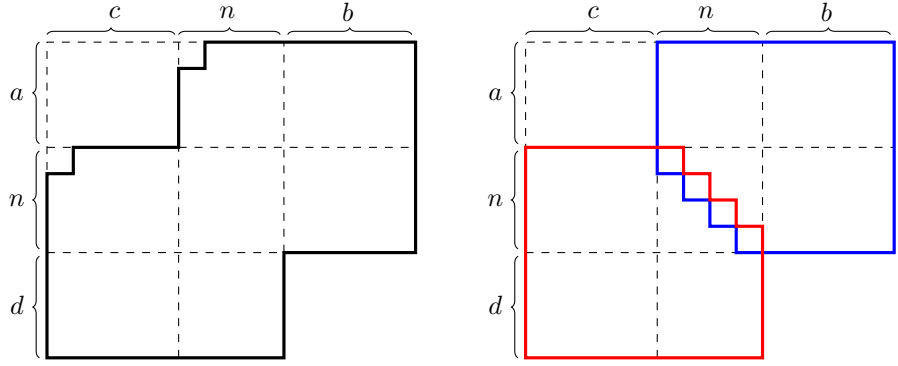


\begin{thm}\label{thm:rho}
Let $\lambda = ((n+b+c)^{n+a},(n+c)^{d})$, $\mu=((c +1), c ^{a -1},1)$ and $\rho=\lambda/\mu$. Then
\begin{multline*}
f^\rho = |\rho|! \frac{\PPhi(n)\PPhi(a)\PPhi(b)\PPhi(c)\PPhi(d)\PPhi(n+a+c)\PPhi(n+b+d)\PPhi(n+a+b+c+d)}{\PPhi(a+c)\PPhi(b+d)\PPhi(n+a+b)\PPhi(n+c+d)\PPhi(2n+a+b+c+d)}\\
\times \frac{ac \big( n(n+a+b+c+d)(n(n+a+b+c+d)+(a+c)(b+d))+bd (a+c-1)(a+c+1) \big)}{(a+c-1)(a+c +1)}.
\end{multline*}
\end{thm}

\begin{proof}
Consider the upper-right half part $\rho^r$ and the lower-left half part $\rho^l$ of $\rho$ divided at the main diagonal with the dented square part filled  as shown in Figure~\ref{fig:decomp}. In other words, $\rho^r = \kappa(a,b,n)$ and $\rho^l = \kappa(c,d,n)'$, where
\[
\kappa(a,b,n) = (n+a+b,n+a+b-1,\dots, b+1)^\ast/(\delta_{a+1})^\ast
\]

Define $f(a,b,n,\mu, t)$ to be the generating function of RPP's $T$ of shape $\kappa(n,a,b)$ such that $\rdiag(T)=\mu$ and
the value in the top-left corner cell is at least $t$ (hence, all the values are 
greater than or equal to $t$), i.e.,
\[
f(a,b,n,\mu, t) := \sum_{\substack{T\in \RPP(\kappa(a,b,n))\\ \rdiag(T)=\mu, ~\min(T)\ge t}}q^{|T|}
= q^{-t\binom{a+1}{2}}\sum_{\substack{T' \in \RPP((\delta_{n+a+1}+(b^{n+a}))^\ast)\\ \rdiag(T')=(\mu_1,\dots, \mu_n, t,\dots, t)}} q^{|T'|}.
\]
The right hand side of the above equation is obtained by filling the skewed part $(\delta_{a+1})^\ast$ by $t$'s.  
By applying Theorem \ref{thm:KS_8.7}, we get 
\begin{align*}
&f(a,b,n,\mu,t) \\
&= \frac{q^{-t\binom{a+1}{2}-\nn(\delta_{n+a+1}+b^{n+a})+|\mu+a\cdot t+\delta_n|}}{\prod_{j=1}^{n+a}(q;q)_{n+a+b-j}}\overline{a}_{(b^{n+a})+\delta_{n+a}}(q^{\mu_1 +n+a-1},\dots, q^{\mu_n +a},q^{t+a-1},\dots, q^t)\\
&= \frac{q^{|\mu|-t\binom{a}{2}-\nn(\delta_{n+a+1})-a\binom{n+a}{2}+\binom{n+a}{2}+b(|\mu|+a(n+t)+\binom{a}{2}+\binom{n}{2})}}{\prod_{j=1}^{n+a}(q;q)_{n+a+b-j}}
\overline{\Delta}(q^{\mu_1 +n+a-1},\dots, q^{\mu_n +a},q^{t+a-1},\dots, q^t)\\
&= \frac{q^{(b+1)|\mu|+at(n+b)+B(a,b,n)}}{\prod_{j=1}^{n+a}(q;q)_{n+a+b-j}}\overline{\Delta}(q^{\mu_1 +n-1},\dots, q^{\mu_n })\overline{\Delta}(q^{a-1},\dots, q^0)\prod_{i=1}^n (q^{\mu_i +n-i}q^{1-t};q)_a,
\end{align*}
where $B(a,b,n)=-\nn(\delta_{n+a+1})+(1-a)\binom{n+a}{2}+abn+(n+b)\binom{a}{2}+(a+b)\binom{n}{2}$. 
Note that
$$
\sum_{\substack{T\in \RPP(\rho^r)\\ \rdiag(T)=\mu, ~\min(T)=0}}q^{|T|} = f(a,b,n,\mu, 0) - f(a,b,n,\mu, 1).
$$
Hence, 
\begin{align*}
&\sum_{T\in \RPP(\rho)}q^{|T|}=\sum_{\mu\in\Par_n} q^{-|\mu|} 
\sum_{\substack{T\in \RPP(\rho^r)\\ \rdiag(T)=\mu, ~\min(T)=0}}q^{|T|}
\sum_{\substack{T\in \RPP(\rho^l)\\ \rdiag(T)=\mu, ~\min(T)=0}}q^{|T|}\\
&= \sum_{\mu \in \Par_n}q^{-|\mu|}( f(a,b,n,\mu, 0) - f(a,b,n,\mu, 1))( f(c,d,n,\mu, 0) - f(c,d,n,\mu, 1))\\
&= \sum_{\mu\in \Par_n} \frac{q^{(b+d+1)|\mu|+B(a,b,n)+B(c,d,n)}\overline{\Delta}(q^{a-1},\dots, q^0)\overline{\Delta}(q^{c-1},\dots, q^0) \overline{\Delta}(q^{\mu+\delta_n})^2}{\prod_{j=1}^{n+a}(q;q)_{n+a+b-j}\prod_{j=1}^{n+c}(q;q)_{n+c+d-j}}\\
&  \times\left(\prod_{i=1}^n (q^{\mu_i +n-i+1};q)_{a}-q^{a (n+b)}\prod_{i=1}^n (q^{\mu_i +n-i};q)_{a}\right) 
\left(\prod_{i=1}^n (q^{\mu_i +n-i+1};q)_{c}-q^{c (n+d)}\prod_{i=1}^n (q^{\mu_i +n-i};q)_{c}\right)\\
&= \frac{q^{B(a,b,n)+B(c,d,n)-(b+d+1)\binom{n}{2}}\overline{\Delta}(q^{a-1},\dots, q^0)\overline{\Delta}(q^{c-1},\dots, q^0)}{(1-q)^n\prod_{j=1}^{n+a}(q;q)_{n+a+b-j}\prod_{j=1}^{n+c}(q;q)_{n+c+d-j} }\\
& \qquad\times \int_{0\le x_1\le \cdots \le x_n\le 1} \prod_{i=1}^n x_i ^{b+d}\overline{\Delta}(X)^2
\left(\prod_{i=1}^n (q x_i ;q)_{a}-q^{a (n+b)}\prod_{i=1}^n (x_i ;q)_{a}\right)\\
& \qquad\qquad\qquad \qquad\qquad\qquad\times \left(\prod_{i=1}^n (q x_i ;q)_{c}-q^{c (n+d)}\prod_{i=1}^n (x_i ;q)_{c}\right)d_q x_1 \cdots d_q x_n,
\end{align*}
where $X=(x_1,\dots, x_n )$ and the last identity follows from Lemma~\ref{lem:gf}. By Lemma~\ref{lem:SYT_RPP}, 
\begin{align}
f^{|\rho|}&=\lim_{q\rightarrow 1}  \frac{q^{B(a,b,n)+B(c,d,n)-(b+d+1)\binom{n}{2}} (q;q)_{|\rho|}\overline{\Delta}(q^{a-1},\dots, q^0)\overline{\Delta}(q^{c-1},\dots, q^0)}{(1-q)^n\prod_{j=1}^{a+n}(q;q)_{n+a+b-j}\prod_{j=1}^{n+c}(q;q)_{n+c+d-j} }\nonumber\\
& \qquad\times \int_{0\le x_1\le \cdots \le x_n\le 1} \prod_{i=1}^n x_i ^{b+d}\overline{\Delta}(X)^2
\left(\prod_{i=1}^n (q x_i ;q)_{a}-q^{a (n+b)}\prod_{i=1}^n (x_i ;q)_{a}\right)\nonumber\\
& \qquad\qquad\qquad \qquad\qquad\qquad\times \left(\prod_{i=1}^n (q x_i ;q)_{c}-q^{c (n+d)}\prod_{i=1}^n (x_i ;q)_{c}\right)d_q x_1 \cdots d_q x_n\nonumber\\
&=|\rho|! \frac{\PPhi(a)\PPhi(b)\PPhi(c)\PPhi(d)}{\PPhi(n+a+b)\PPhi(n+c+d)}  \int_{0\le x_1\le \cdots \le x_n\le 1} \prod_{i=1}^n x_i ^{b+d}\overline{\Delta}(X)^2\nonumber\\
&\qquad\qquad\qquad\qquad \times \lim_{q\rightarrow 1}\frac{\prod_{i=1}^n (q x_i ;q)_{a}-q^{a (n+b)}\prod_{i=1}^n (x_i ;q)_{a}}{1-q}\nonumber\\
&\qquad\qquad\qquad\qquad \times \lim_{q\rightarrow 1}\frac{\prod_{i=1}^n (q x_i ;q)_{c}-q^{c (n+d)}\prod_{i=1}^n (x_i ;q)_{c}}{1-q} dx_1 \cdots dx_n.\label{eqn:syt}
\end{align}
Let us calculate the limits separately:
\begin{align*}
&  \lim_{q\rightarrow 1}\frac{\prod_{i=1}^n (q x_i ;q)_{a}-q^{a (n+b)}\prod_{i=1}^n (x_i ;q)_{a}}{1-q}\\
& = \lim_{q\rightarrow 1} \prod_{i=1}^n (q x_i ;q)_{a-1}\frac{\prod_{i=1}^n (1-q^a x_i)-q^{a(n+b)}\prod_{i=1}^n (1-x_i)}{1-q}\\
&= \prod_{i=1}^n (1-x_i)^{a-1} \lim_{q\rightarrow 1}  \frac{\prod_{i=1}^n (1-q^a x_i)\sum_{j=1}^n \frac{-a q^{a-1}x_j}{1-q^a x_j}-a(b+n)q^{a(n+b)-1}\prod_{i=1}^n (1-x_i)}{-1}\\
&=a\prod_{i=1}^n (1-x_i)^a \left(\sum_{j=1}^n \frac{ x_j}{1-x_j}+(n+b) \right).
\end{align*}
Now we compute the integral part applying this limit computation, with the change of variables $x_i \mapsto 1-x_i$. Then we have 
\begin{align*}
& \int_{0\le x_1\le \cdots \le x_n\le 1} \prod_{i=1}^n x_i ^{b+d}\overline{\Delta}(X)^2
 \lim_{q\rightarrow 1}\frac{\prod_{i=1}^n (q x_i ;q)_{a}-q^{a (n+b)}\prod_{i=1}^n (x_i ;q)_{a}}{1-q}\\
&\qquad\qquad\qquad\qquad \times  \lim_{q\rightarrow 1}\frac{\prod_{i=1}^n (q x_i ;q)_{c}-q^{c (n+d)}\prod_{i=1}^n (x_i ;q)_{c}}{1-q} dx_1 \cdots dx_n\\
& = a c \int_{0\le x_1\le \cdots \le x_n\le 1} \prod_{i=1}^n x_i ^{a+c}(1-x_i )^{b+d}\overline{\Delta}(X)^2\\
&\qquad\qquad\qquad\qquad \times \left( \sum_{j=1}^n \frac{1-x_j}{x_j}+(n+b) \right) \left( \sum_{j=1}^n \frac{1-x_j}{x_j}+(n+d) \right)dx_1 \cdots dx_n\\
& = a c  \int_{0\le x_1\le \cdots \le x_n\le 1} \prod_{i=1}^n x_i ^{a+c}(1-x_i )^{b+d}\overline{\Delta}(X)^2\\
& \qquad\qquad\qquad\qquad \times \left( \prod_{i=1}^n x_i ^{-1} s_{(1^{n-1})}(X)+b \right) \left( \prod_{i=1}^n x_i ^{-1} s_{(1^{n-1})}(X)+d \right)dx_1 \cdots dx_n\\
&= a c   \int_{0\le x_1\le \cdots \le x_n\le 1}s_{(2^{n-2},1,1)}(X) \prod_{i=1}^n x_i ^{a+c-2}(1-x_i )^{b+d}\overline{\Delta}(X)^2dx_1 \cdots dx_n\\
& \quad +  a c   \int_{0\le x_1\le \cdots \le x_n\le 1}s_{(2^{n-1})}(X) \prod_{i=1}^n x_i ^{a+c-2}(1-x_i )^{b+d}\overline{\Delta}(X)^2dx_1 \cdots dx_n\\
& \quad +  a c  (b+d) \int_{0\le x_1\le \cdots \le x_n\le 1}s_{(1^{n-1})}(X) \prod_{i=1}^n x_i ^{a+c-1}(1-x_i )^{b+d}\overline{\Delta}(X)^2dx_1 \cdots dx_n\\
& \quad +  a b c d   \int_{0\le x_1\le \cdots \le x_n\le 1}\prod_{i=1}^n x_i ^{a+c}(1-x_i )^{b+d}\overline{\Delta}(X)^2dx_1 \cdots dx_n.
\end{align*}
Here, Pieri's rule is used for $X=(x_1,\dots,x_n)$ :
\[
s_{(1^{n-1})}(X) ^2 = s_{(1^{n-1})}(X) e_{n-1}(X)
= s_{(2^{n-2},1,1)}(X) +s_{(2^{n-1})}(X).
\]

To compute the last four Selberg-type integrals, we note the result of Warnaar \cite[Corollary 1.3]{Warnaar2005}, when $k=1$ and $q\rightarrow 1$: 
\begin{multline*}
\qquad\qquad\int_{[0,1]^n}s_\lambda (X) \prod_{i=1}^n x_i ^{\alpha -1} (1-x_i)^{\beta -1} \overline{\Delta}(X) ^2 dx_1 \cdots dx_n\\
=\prod_{(i,j)\in \lambda}\frac{n-i+j}{h_\lambda (i,j)}\prod_{i=1}^n \frac{(\alpha+n-i+\lambda_i -1)! (\beta +i-2)! i!}{(\alpha+ \beta +2n-i-2+\lambda_i)!}.\qquad\qquad
\end{multline*}
 By using Warnaar's result, we get 
\begin{multline}\label{eqn:int1}
a c  \int_{0\le x_1\le \cdots \le x_n\le 1}s_{(2^{n-2},1,1)}(X) \prod_{i=1}^n x_i ^{a+c-2}(1-x_i )^{b+d}\overline{\Delta}(X)^2dx_1 \cdots dx_n\\
 = a c  \binom{n}{2} \frac{(n+a+b+c+d +1)! }{(a+c +1)! } \frac{\PPhi(n+1)\PPhi(n+a+c)\PPhi(n+b+d)\PPhi(n+a+b+c+d-1)}{\PPhi(a+c-1)\PPhi(b+d)\PPhi(2n+a+b+c+d)},
\end{multline}
\begin{multline}\label{eqn:int2}
 a c  \int_{0\le x_1\le \cdots \le x_n\le 1}s_{(2^{n-1})}(X) \prod_{i=1}^n x_i ^{a+c-2}(1-x_i )^{b+d}\overline{\Delta}(X)^2dx_1 \cdots dx_n\\
 = a c  \binom{n+1}{2} \frac{(a+c -2)! }{(n+a+b+c+d-2)! }\frac{\PPhi(n+1)\PPhi(n+a+c)\PPhi(n+b+d)\PPhi(n+a+b+c+d+1)}{\PPhi(a+c+1)\PPhi(b+d)\PPhi(2n+a+b+c+d)},
\end{multline}
\begin{multline}\label{eqn:int3}
 a c  (b+d) \int_{0\le x_1\le \cdots \le x_n\le 1}s_{(1^{n-1})}(X) \prod_{i=1}^n x_i ^{a+c-1}(1-x_i )^{b+d}\overline{\Delta}(X)^2dx_1 \cdots dx_n\\
 = ac n(b+d) \frac{(a+c -1)! }{(n+a+b+c+d-1)!}\frac{\PPhi(n+1)\PPhi(n+a+c)\PPhi(n+b+d)\PPhi(n+a+b+c+d+1)}{\PPhi(a+c+1)\PPhi(b+d)\PPhi(2n+a+b+c+d)},
 \end{multline}
 \begin{multline}\label{eqn:int4}
a b c d  \int_{0\le x_1\le \cdots \le x_n\le 1}\prod_{i=1}^n x_i ^{a+c}(1-x_i )^{b+d}\overline{\Delta}(X)^2dx_1 \cdots dx_n\\
 = a b c d \frac{\PPhi(n+1)\PPhi(n+a+c)\PPhi(n+b+d)\PPhi(n+a+b+c+d)}{\PPhi(a+c)\PPhi(b+d)\PPhi(2n+a+b+c+d)}.
 \end{multline}
 Adding the above four results gives us  
 \begin{multline*}
 \eqref{eqn:int1}+ \eqref{eqn:int2}+ \eqref{eqn:int3}+ \eqref{eqn:int4}\\
 = \frac{a c }{(a+c -1)(a+c +1)}\frac{\PPhi(n+1)\PPhi(n+a+c)\PPhi(n+b+d)\PPhi(n+a+b+c+d)}{\PPhi(a+c)\PPhi(b+d)\PPhi(2n+a+b+c+d)}\\
  \times \big( n(n+a+b+c+d)(n(n+a+b+c+d)+(a+c)(b+d))+b d  (a+c  -1)(a+c  +1) \big).
 \end{multline*}
 We obtain the formula for $f^{|\rho|}$ by replacing the integration part in \eqref{eqn:syt} by the above computation.
\end{proof}

\section{Enumeration of standard Young tableaux of shifted skew shape}\label{sec:conj2}

In this section we prove Conjecture~\ref{conj:MPP3-2}, which is restated below.

\begin{thm}\label{thm:pi}
For $\pi=\VV(n,a,b,m)$, the number $g^\pi$ of standard Young tableaux of shape $\pi$ is 
\begin{equation*}
g^\pi = \frac{|\pi|!}{2^a} \cdot \frac{\PPhi(n+2a)\PPhi(a)}{\PPhi(2a)\PPhi(n+a)} \cdot \frac{\gimel(2a)\gimel(n)}{\gimel(n+2a)}
\prod_{(i,j)\in \lambda\backslash D}\frac{1}{h_{\lambda^*} (i,j)},
\end{equation*}
where $\gimel (n) = \prod_{i=1}^{\flr{n/2}} (n-2i)!$,
$\lambda = ((n+a+b, n+a+b-1,\dots, b+1)+(m-1) \delta_{n+a})^\ast$ and $D$ is the set of cells $(i,n+j)$ with $1\le i\le j\le n$.
\end{thm}

\begin{proof}
Firstly, by computing the shifted hook lengths of the cells in $\lambda\backslash D$ explicitly, 
we can rewrite the conjectured formula for $g^\pi$ as 
\begin{multline}\label{eqn:conjpi}
g^\pi = \frac{|\pi|!}{2^a} \cdot \frac{\PPhi(n+2a)\PPhi(r)}{\PPhi(2a)\PPhi(n+a)} \cdot \frac{\gimel(2a)\gimel(n)}{\gimel(n+2a)}\\
\times \frac{m^{\binom{n+a}{2}}\PPhi (n+a)\prod_{i=1}^a \prod_{j=0}^{i-1}(2(b+1)+(n+i+j-1)m)}{\prod_{i=0}^{n+a-1}(b+mi)! \prod_{i=0}^{n+a-1}(b+1+mi) \prod_{i=1}^{n+a-1}\prod_{j=0}^{i-1}(2(b+1)+(i+j)m)}.
\end{multline}

To utilize the generating function formula for the semistandard Young tableaux of shifted shapes 
with fixed diagonals given in Corollary \ref{cor:gfs}, we fill the top row of the skewed part $(\delta_{a+1})^\ast$ by $0$'s,  the second row by $1$'s, and so on. 
If we say we fix the diagonal cells of $\lambda$ by $\nu :=(\mu_1 +n-1, \dots, \mu_{n-1}+1,\mu_n,a-1,\dots, 1,0)$ for some $\mu\in \Par _n$, then we have 
\begin{align*}
\sum_{\substack{T\in \SSYT (\pi)\\ \rdiag(T)=\nu}}q^{|T|}&=q^{-\binom{a+1}{3}}\sum_{\substack{T\in\SSYT(\lambda)\\ \rdiag(T)=\nu }}q^{|T|}\\
&= q^{-\binom{a+1}{3}}\frac{q^{\binom{a}{2}+|\mu+\delta_n |}}{\prod_{j=1}^{n+a}(q;q)_{b+m(n+a-j)}} \overline{a}_{(b^{n+a})+m\delta_{n+a}}(q^\nu).
\end{align*}
Summing this up over all partitions with at most $n$ parts gives 
\begin{align*}
&\sum_{T\in \SSYT (\pi)}q^{|T|} =\sum_{\mu\in \Par_n}\sum_{\substack{T\in \SSYT (\pi)\\ \rdiag(T)=\nu}}q^{|T|} \\
&= \frac{q^{-\binom{a}{3}}}{(1-q)^n \prod_{i=1}^{n+a}(q;q)_{b+m(n+a-j)}}
\int_{0\le x_1 \le \cdots \le x_n \le 1}\overline{a}_{(b^{n+a})+m\delta_{n+a}}(x_1,\dots, x_n,q^{a-1},\dots, q^0)d_q X.
\end{align*}
To simplify the integrand, we note the following :
\[
 a_{(b^n) +\lambda}(x_1,\dots, x_n) = \det (x_i ^{b+\lambda_j})=\prod_{i=1}^n x_i ^b \cdot \det (x_i ^{\lambda_j})
= \prod_{i=1}^n x_i ^b\cdot  a_{\lambda}(x_1,\dots, x_n)
\]
and 
$$a_{m\lambda}(x_1,\dots, x_n) =\det (x_i ^{m\lambda_j})=a_{\lambda}(x_1 ^m,\dots, x_n ^m).
$$
Thus we can rewrite the integrand as 
$$
 \overline{a}_{(b^{a+n})+m\delta_{n+a}}(x_1,\dots, x_n, q^{a-1},\dots, q^0)
= q^{b\binom{a}{2}}\prod_{i=1}^n x_i ^{b} \cdot \overline{\Delta}(x_1 ^m ,\dots, x_n ^m, q^{m(a-1)},\dots, q^{m\cdot 0}).$$
So far we have 
\begin{multline*}
\sum_{T\in \SSYT (\pi)}q^{|T|} =
\frac{q^{b\binom{a}{2}-\binom{a}{3}}}{(1-q)^n \prod_{i=1}^{n+a}(q;q)_{b+m(n+a-j)}}\\
\times \int_{0\le x_1 \le \cdots \le x_n\le 1 }\prod_{i=1}^n x_i ^{b} \cdot \overline{\Delta}(x_1 ^m ,\dots, x_n ^m, q^{m(a-1)},\dots, q^{m\cdot 0}) d_q x_1\cdots d_q x_n.
\end{multline*}
We apply Lemma \ref{lem:cov} to make the change of variables $x_i ^m \mapsto x_i$, and by letting $p=q^m$ we get 
\begin{multline*}
\sum_{T\in \SSYT (\pi)}q^{|T|} =
\frac{q^{b\binom{a}{2}-\binom{a}{3}}}{(1-p)^n \prod_{i=1}^{n+a}(q;q)_{b+m(n+a-j)}}\\
\times \int_{0\le x_1 \le \cdots \le x_n\le 1 }\prod_{i=1}^n x_i ^{\frac{b}{m}+\frac{1-m}{m}} \cdot \overline{\Delta}(x_1 ,\dots, x_n ,p^{a-1},\dots, p^0) d_p x_1\cdots d_p x_n.
\end{multline*}
Note that 
$$\overline{\Delta}(x_1 ,\dots, x_n ,p^{a-1},\dots, p^0 )  = p^{\binom{a}{3}+n\binom{a}{2}}\prod_{i=1}^{a-1}(p;p)_i \cdot \prod_{i=1}^n (p^{1-a}x_i ;p)_a \cdot\overline{\Delta}(x_1,\dots, x_n),$$
and so we have 
\begin{multline*}
\sum_{T\in \SSYT (\pi)}q^{|T|} =
\frac{q^{b\binom{a}{2}-\binom{a}{3}+m\binom{a}{3}+mn\binom{a}{2}}\prod_{i=1}^{a-1}(q^m;q^m)_i }{(1-q^m)^n \prod_{i=1}^{n+a}(q;q)_{b+m(n+a-j)}}\\
\times \int_{0\le x_1 \le \cdots \le x_n\le 1 }\prod_{i=1}^n x_i ^{\frac{b-m+1}{m}}(p^{1-a} x_i;p)_a \cdot \overline{\Delta}(x_1 ,\dots, x_n ) d_p x_1\cdots d_p x_n.
\end{multline*}
Given this, by Lemma~\ref{lem:SYT_RPP}, we can get the number of standard Young tableaux of shape $\pi$ by 
\begin{align*}
g^\pi &= \lim_{q\rightarrow 1}\left( (q;q)_{|\pi|} \sum_{T\in \SSYT (\pi)}q^{|T|} \right)\\
&= \frac{|\pi|! \cdot m^{\binom{a}{2}}\prod_{i=1}^{a-1}i!}{m^n \prod_{j=0}^{n+a-1}(b+m j)!}
\int_{0\le x_1 \le \cdots \le x_n \le 1}\prod_{i=1}^n x_i ^{\frac{b-m+1}{m}}(1-x_i)^a \cdot \overline{\Delta}(x_1,\dots, x_n) dx_1\cdots d x_n.
\end{align*}
Notice that the integral is a special case of the well-known \emph{Selberg integral} :
\begin{align}\label{eqn:Selbergint}
&\int_{0\le x_1 \le \cdots \le x_n \le 1}\prod_{i=1}^n x_i ^{\frac{b+1}{m}-1}(1-x_i)^a \cdot \overline{\Delta}(x_1,\dots, x_n) dx_1\cdots d x_n \nonumber\\
&=\frac{1}{n!}\prod_{j=1}^n \frac{\Gamma (\frac{b+1}{m}+\frac{1}{2}(j-1))\Gamma(a+1+\frac{1}{2}(j-1))\Gamma(1+\frac{1}{2}j)}{\Gamma(\frac{b+1}{m}+a+1+\frac{1}{2}(n+j-2))\Gamma(1+\frac{1}{2})}\nonumber\\
&=\begin{cases}
\dfrac{m^{N+2N(N+a)}2^N\PPhi(2N+2a)\gimel(2a)\gimel(2N)}{\PPhi(2a)\gimel (2N+2a)\prod_{i=0}^{N-1}(b+1+mi)\prod_{j=1}^{2N}\prod_{i=0}^{N+a-1}(2(b+1)+(2i+j)m)}, &\text{ if }n=2N,\\
\dfrac{m^{(2N+1)(N+a+1)}2^{2N+1}\PPhi(2N+2a+1)\gimel(2a)\gimel(2N+1)}{\PPhi(2a)\gimel(2N+2a+1)\prod_{j=1}^{2N+1}\prod_{i=0}^{N+a}(2(b+1)+(2i+j-1)m)}, & \text{ if } n=2N+1,
\end{cases}
\end{align}
where \eqref{eqn:Selbergint} is obtained by explicitly computing the gamma function values. By replacing the integral part in $g^\pi$ by \eqref{eqn:Selbergint},
we get 
$$g^\pi = \frac{|\pi|! m^{\binom{a+n}{2}}\PPhi(a) \PPhi(n+2a)\gimel(2a)\gimel(n)}{\PPhi(2a)\gimel(n+2a)}\cdot P(a,b,n,m),$$
where 
\begin{align*}
&P(a,b,n,m)\\=
&\begin{cases}
\dfrac{2^N}{\prod_{j=0}^{2N+a-1}(b+mj)! \prod_{i=0}^{N-1}(b+1+mi)\prod_{j=1}^{2N}\prod_{i=0}^{N+a-1}(2(b+1)+(2i+j)m)}, & \text{ if } n=2N,\\
\dfrac{2^{2N+1}}{\prod_{j=0}^{2N+a}(b+mj)! \prod_{j=1}^{2N+1}\prod_{i=0}^{N+a}(2(b+1)+(2i+j-1)m)}, & \text{ if } n=2N+1.
\end{cases}
\end{align*}
By comparing this formula to \eqref{eqn:conjpi}, to verify Morales, Pak and Panova conjecture, we only need to prove 
$$P(a,b,n,m)= \frac{\prod_{i=1}^r\prod_{j=0}^{i-1}(2(b+1)+(n+i+j-1)m)}{2^a \prod_{i=0}^{n+a-1}(b+mi)! \prod_{i=0}^{n+a-1}(b+1+mi) \prod_{i=1}^{n+a-1}\prod_{j=0}^{i-1}(2(b+1)+(i+j)m)}.$$
It is not very hard to check that they are two different ways of expressing multiplications of the same set of factors. 
This finishes the proof.
\end{proof}

\section{Generalized MacMahon's box theorem using $q$-integrals}
\label{sec:GMBT}

In this section, we prove Theorem~\ref{thm:GMBT1}, which is restated as follows.

\begin{thm}\label{thm:GMBT}
For $\pi=\MM(n,a,b,c,d,1)$ and an integer $N\ge0$, we have 
\begin{multline*}
  s_\pi(1,q,q^2,\dots,q^N) =q^{\sum_{(i,j)\in \lambda/(c^a)}(\lambda_j ' -i)}
 \prod_{i=1}^{b}(q^{N-a +1+i};q)_{a} \prod_{i=1}^{d} (q^{N+2-i};q)_{c}\\
\times \prod_{i=1}^n (q^{N-n-a-d+1+i};q)_{n+a+b+c+d} \prod_{i=1}^n \prod_{j=1}^{a}\prod_{k=1}^{c}\frac{1-q^{i+j+k-1}}{1-q^{i+j+k-2}}
 \prod_{(i,j)\in \lambda\setminus(0^n,c^a)}\frac{1}{1-q^{h_\lambda(i,j)}}.
\end{multline*}
\end{thm}

\begin{proof}
Recall from Corollary \ref{cor:gfs} that
\begin{equation}\label{eqn:ssytgf}
\sum_{\substack{T\in \SSYT((\delta_{n+1}+\lambda)^\ast) \\ \rdiag (T)=\mu+\delta_n}} q^{|T|} 
= \frac{q^{|\mu+\delta_n|}}{\prod_{j=1}^n (q;q)_{\lambda_j +n -j}} \overline{a}_{\lambda+\delta_n }(q^{\mu+\delta_n}).
\end{equation}
We first consider the upper right half part of $\pi$, divided at the main diagonal, that is, $\pi^{rh} := (n+a+b, n+a+b-1,\dots, b+1)^\ast/(\delta_{a+1})^\ast$. To utilize \eqref{eqn:ssytgf}, we fill the skewed $(\delta_{a+1})^\ast$ part 
by $0$'s in the first row, $1$'s in the second row, and so on, and by $(a-1)$ in the $a$-th row. 
Similarly, we attach $(\delta_{b +1}) ^\ast$ below the last row and fill it with $(N+1)$'s in the first row, $(N+2)$'s in the second row, and so on, and with $(N+b)$ in the $b$-th row. 
Then 
\begin{align*}
\sum_{\substack{T\in \SSYT(\pi^{rh}) \\ \rdiag(T)=\mu+\delta_n \\ \min(T)\ge 0, ~\max(T)\le N }}q^{|T|}
&=q^{-\binom{a +1}{3}-\binom{b +2}{3}-N\binom{b +1}{2}}
\sum_{\substack{T\in \SSYT((\delta_{n+a+b+1})^\ast)\\ \rdiag(T)= ( N+b, \dots, N+1, \mu_1 +n-1,\dots, \mu_n, a-1,\dots, 1,0)}}q^{|T|}\\
&= \frac{q^{-\binom{a}{3}-\binom{b +1}{3}-N\binom{b}{2}+|\mu|+\binom{n}{2}}}{\prod_{j=1}^{n+a+b-1}(q;q)_j}\overline{\Delta}(q^{N+b},\dots, q^{N+1},q^{\mu +\delta_n},q^{a-1},\dots, q, 1).
\end{align*}
On the other hand, to deal with the lower-left half of $\pi$ using \eqref{eqn:rstgf} in Corollary \ref{cor:gfs}, 
we reflect the lower-left half of $\pi$ along the main diagonal and denote that part by $\pi^{lh}$. That is, 
$\pi^{lh} =  (n+c+d, n+c+d-1,\dots, d+1)^\ast/(\delta_{c+1})^\ast$.
Note that since we have reflected the diagram along the main diagonal, the fillings satisfying the conditions of semistandard Young tableaux becomes 
row strict tableaux.
Keeping this in mind, again attach 
$(\delta_{c +1})^\ast$ in front and $(\delta_{d +1})^\ast$ below the last row, and fill the cells by $(-c)$ in the first column, 
$(-c +1)$'s in the second column, and so on, and by $(-1)$'s in the $c$-th column of the $(\delta_{c+1})^\ast$ part attached in front, and 
by $(N-d+1)$ in the first column of the  $(\delta_{d +1})^\ast$ part, by $(N-d+2)$'s in the second column, and so on, and by $N$'s in the last column.
Then we have 
\begin{align*}
&\sum_{\substack{T\in \RST(\pi^{lh}) \\ \rdiag(T)=\mu+\delta_n \\ \min(T)\ge 0, ~\max(T)\le N }}q^{|T|}\\
&=q^{\binom{c +2}{3}+\binom{d +1}{3}-N\binom{d +1}{2}}
\sum_{\substack{T\in \RST((\delta_{n+c+d +1})^\ast)\\ \rdiag(T)= (N, N-1, \dots, N-d+1,\mu_1 +n-1,\dots, \mu_n,-1,\dots, -c )}}q^{|T|}\\
&= \frac{q^{\binom{c +1}{3}+\binom{d}{3}-N\binom{d}{2}+\binom{n+c+d+1}{3}+|\mu|+\binom{n}{2}}}{\prod_{j=1}^{n+c+d-1}(q;q)_j}\overline{\Delta}(q^N, q^{N-1},\dots, q^{N-d+1},q^{\mu+\delta_n},q^{-1},\dots, q^{-c}).\end{align*}
Hence, 
\begin{align*}
&\sum_{\substack{T\in\SSYT(\pi)\\ \min(T)\ge 0,~\max(T) \le N}} q^{|T|} =\sum_{\mu\in \Par_n}q^{-|\mu|-\binom{n}{2}}
\sum_{\substack{T\in \SSYT(\pi^{rh}) \\ \rdiag(T)=\mu+\delta_n \\ \min(T)\ge 0, ~\max(T)\le N }}q^{|T|} \sum_{\substack{T\in \SSYT(\pi^{lh}) \\ \rdiag(T)=\mu+\delta_n \\ \min(T)\ge 0, ~\max(T)\le N }}q^{|T|}\\
 &=  \frac{q^{-\binom{a}{3}-\binom{b +1}{3}+\binom{c +1}{3}+\binom{d}{3}-N\left(\binom{b}{2}+\binom{d}{2}\right)+\binom{n+c+d+1}{3}+|\mu|+\binom{n}{2}}}{\prod_{j=1}^{n+a+b-1}(q;q)_j \prod_{j=1}^{n+c+d-1}(q;q)_j}\\
 &\quad \times \overline{\Delta}(q^{N+b},\dots, q^{N+1},q^{\mu +\delta_n},q^{a-1},\dots, q, 1) \overline{\Delta}(q^N, q^{N-1},\dots, q^{N-d+1},q^{\mu+\delta_n},q^{-1},\dots, q^{-c})\\
 &= \frac{q^{-\binom{a}{3}-\binom{b +1}{3}+\binom{c +1}{3}+\binom{d}{3}-N\left(\binom{b}{2}+\binom{d}{2}\right)+\binom{n+c+d+1}{3}}}{(1-q)^n \prod_{j=1}^{n+a+b-1}(q;q)_j \prod_{j=1}^{n+c+d-1}(q;q)_j}\\
 & \quad\times \int_{q^{b-s_2 +1}\le x_1 \le \cdots \le x_n\le q^{a+r_1}} 
  \overline{\Delta}(q^{N+b},\dots, q^{N+1},x_1,\dots, x_n ,q^{a-1},\dots, q, 1)\\
  & \qquad\qquad\qquad\times  \overline{\Delta}(q^N, q^{N-1},\dots, q^{N-d+1},x_1,\dots, x_n,q^{-1},\dots, q^{-c})d_q x_1 \cdots d_q x_n.
\end{align*}
Note that 
\begin{multline*}
  \overline{\Delta}(q^{N+b},\dots, q^{N+1},x_1,\dots, x_n ,q^{a-1},\dots, q, 1)\\
 = q^{\binom{a}{3}+\binom{b+1}{3}+b\binom{a}{2}+n\left( \binom{a}{2}+\binom{b+1}{2}\right)+N\left( \binom{b}{2}+bn\right)}\overline{\Delta}(x_1,\dots, x_n)\\
 \times \prod_{i=1}^{a -1}(q;q)_i \prod_{i=1}^{b -1} (q;q)_i \prod_{i=1}^{b}( q^{N-a +1+i};q)_{a} \prod_{i=1}^n (q^{-a+1} x_i ;q)_{a}(q^{-N-b}x_i ;q)_{b}
\end{multline*}
and 
\begin{multline*}
  \overline{\Delta}(q^N, q^{N-1},\dots, q^{N-d+1},x_1,\dots, x_n,q^{-1},\dots, q^{-c})\\
 = q^{-2\binom{c+1}{3}-\frac{1}{6}d(d-1)(2d-1)-d\binom{c+1}{2}-n\left( \binom{c+1}{2}+\binom{d}{2}\right)+N\left(\binom{d}{2}+dn \right) }\overline{\Delta}(x_1,\dots, x_n)\\
 \times \prod_{i=1}^{c-1}(q;q)_i \prod_{i=1}^{d -1} (q;q)_i  \prod_{i=1}^{d}( q^{N+2-i};q)_{c}\prod_{i=1}^n (q x_i ;q)_{c}(q^{-N}x_i ;q)_{d}.
\end{multline*}
Applying the above computations gives 
\begin{align*}
&\sum_{\substack{T\in\SSYT(\pi)\\ \min(T)\ge 0,~\max(T) \le N}} q^{|T|}=\\
&= q^{-\binom{c+1}{3}-\binom{d+1}{3}+b\binom{a}{2}-d\binom{c+1}{2}+\binom{n+c+d+1}{3}+n\left(\binom{a}{2}+\binom{b+1}{2}-\binom{c+1}{2}-\binom{d}{2} \right)+nN(b+d)}\\
&\qquad \times \frac{\PPhi_q (a)\PPhi_q (b)\PPhi_q(c)\PPhi_q(d)\prod_{i=1}^{b}(q^{N-a+1+i};q)_{a}\prod_{i=1}^{d}(q^{N+2-i};q)_{c}}{(1-q)^n\PPhi_q(n+a+b)\PPhi_q(n+c+d)}\\
&\qquad  \times  \int_{q^{N-d +1}\le x_1 \le \cdots \le x_n\le q^{a}}  \overline{\Delta}(x_1,\dots, x_n) ^2
\prod_{i=1}^n (q x_i / q^{a} ;q)_{a+c }(q x_i /q^{N+b+1} ;q)_{b+d}  d_q x_1 \cdots d_q x_n.
\end{align*}
Note that the factor $(qx_i/q^{N+b+1};q)_{b+d}$ in the integrand becomes zero for the values $q^{N-b +1}\le x_1\le q^{N+b}$. Hence we can change the lower end of the integral by $q^{N+b+1}$. Then the integral is a special case of the $q$-Selberg integral which was conjectured by Askey\cite{Askey1980} and proved by Habsieger\cite{Habsieger1988}, Kadell\cite{Kadell1988b} and Evans\cite{Evans1992} :
\begin{align}
& \int_{\mathsf{a} \le X\le \mathsf{b}}\Delta (X) ^2 \prod_{i=1}^n \left(\frac{q x_i}{\mathsf{a}}\right)_{\alpha -1}\left(\frac{q x_i}{\mathsf{b}} \right)_{\beta -1}d_q x_1 \cdots d_q x_n\nonumber\\
&= (-1)^{\binom{n}{2}}q^{\binom{n}{3}}\prod_{i=0}^{n-1} \frac{\Gamma_q (\alpha+i) \Gamma_q (\beta +i) \Gamma_q (i+1) \left( \frac{\mathsf{a}}{\mathsf{b}}\right)_{\beta +i}\left( \frac{\mathsf{b}}{\mathsf{a}}\right)_{\alpha+i}(\mathsf{ab})^{i+1}}{\Gamma_q (\alpha+\beta+n+i-1)(\mathsf{a}-\mathsf{b})}\nonumber\\
&=  (-1)^{\binom{n}{2}}q^{\binom{n}{3}}\prod_{i=0}^{n-1} \frac{(1-q)^{n-2i} (q;q)_{\alpha +i-1}(q;q)_{\beta +i-1}(q;q)_i  \left( \frac{\mathsf{a}}{\mathsf{b}}\right)_{\beta +i}\left( \frac{\mathsf{b}}{\mathsf{a}}\right)_{\alpha+i}(\mathsf{ab})^{i+1}}{(\mathsf{a}-\mathsf{b})(q;q)_{\alpha + \beta +n+i-2}}.\label{eqn:qSelbergint}
\end{align}
By specializing $\mathsf{a}=q^{N+b +1}$, $\mathsf{b}=q^{a}$, $\alpha=b+d+1$, $\beta = a+c +1$, we get 
\begin{align*}
&  \int_{q^{N+b+1}\le x_1 \le \cdots \le x_n\le q^{a}}  \overline{\Delta}(x_1,\dots, x_n) ^2
\prod_{i=1}^n (q x_i / q^{a} ;q)_{a+c }(q x_i /q^{N+b+1} ;q)_{b+d}  d_q x_1 \cdots d_q x_n\\
&\qquad = q^{\sum_{i=1}^{n-1}i^2 +2a\binom{n+1}{2}+(a-b-N)n(b+d)+(b+d)\binom{n}{2}+n\binom{b+d}{2}}\\
&\qquad\qquad\times (1-q)^n \prod_{i=1}^n (q^{N-n-a-d+1+i};q)_{n+a+b+c+d}\\
&\qquad\qquad\times \frac{\PPhi_q(n)\PPhi_q(n+a+c)\PPhi_q(n+b+d)\PPhi_q(n+a+b+c+d)}{\PPhi_q(a+c)\PPhi_q(b+d)\PPhi_q(2n+a+b+c+d)}.
\end{align*}
By putting the result of evaluating the Selberg-type integral and simplifying the $q$-power, we get 
\begin{multline*}
\sum_{\substack{T\in\SSYT(\pi)\\ \min(T)\ge 0,~\max(T) \le N}} q^{|T|}=\\
q^{a\binom{n+d}{2}+b\binom{n+a}{2}+n\binom{n+a+d}{2}}
 \prod_{i=1}^{b}(q^{N-a +1+i};q)_{a} \prod_{i=1}^{d} (q^{N+2-i};q)_{c}\prod_{i=1}^n (q^{N-a-d-n+1+i};q)_{n+a+b+c+d}\\
 \times \frac{\PPhi_q (n)\PPhi_q (a)\PPhi_q (b)\PPhi_q (c)\PPhi_q (d)\PPhi_q (n+a+c)\PPhi_q (n+b+d)\PPhi_q (n+a+b+c+d) }{\PPhi_q (a+c)\PPhi_q (b+d)\PPhi_q (n+a+b)\PPhi_q (n+c+d)\PPhi_q (2n+a+b+c+d)}.
\end{multline*}
Note that 
$$a\binom{n+d}{2}+b\binom{n+a}{2}+n\binom{n+a+d}{2}=\sum_{(i,j)\in \lambda/(c^{a})}(\lambda_j ' -i).$$
Also noting that 
$$ \frac{\PPhi_q (a)\PPhi_q (c)\PPhi_q (n)\PPhi_q (n+a+c)}{\PPhi_q (a+c)\PPhi_q (n+a)\PPhi_q (n+c)}= \prod_{i=1}^n \prod_{j=1}^{a}\prod_{k=1}^{c}\frac{1-q^{i+j+k-1}}{1-q^{i+j+k-2}}$$
and considering the hook lengths of the cells in $\lambda\setminus(0^n,c^a)$, we can rewrite the result as 
\begin{align*}
s_\pi(1,q,q^2,\dots,q^N) &=\sum_{\substack{T\in\SSYT(\pi)\\ \min(T)\ge 0,~\max(T) \le N}} q^{|T|}\\
&= q^{ \sum_{(i,j)\in \lambda/(c^{a})}(\lambda_j ' -i)}\\
&\qquad \times \prod_{i=1}^{b}(q^{N-a +1+i};q)_{a} \prod_{i=1}^{d} (q^{N+2-i};q)_{c}\prod_{i=1}^n (q^{N-a-d-n+1+i};q)_{n+a+b+c+d}\\
&\qquad\times \prod_{i=1}^n \prod_{j=1}^{a}\prod_{k=1}^{c}\frac{1-q^{i+j+k-1}}{1-q^{i+j+k-2}}
\cdot \prod_{(i,j)\in \lambda\setminus(0^n,c^a)}\frac{1}{1-q^{h_\lambda(i,j)}}.
\end{align*}
\end{proof}

\section{Skew trace generating function}
\label{sec:trace}

In this section, we prove Theorem~\ref{thm:trace}, which is restated as follows. 
\begin{thm}
Let $\pi=\MM(n,a,b,c,d,m)$. Then
\begin{multline*}
\sum_{T\in\SSYT(\pi)} x^{\tr(T)}q^{|T|} =
x^{n a +\binom{n}{2}}q^{\sum_{(i,j)\in \lambda/(c^a)} (\lambda_j '-i)}\\
\times  \prod_{i=1}^n \prod_{j=1}^{a}\prod_{k=1}^{c}\frac{1-q^{m(i+j+k-1)}}{1-q^{m(i+j+k-2)}}
\prod_{(i,j)\in \lambda\setminus(0^n, c^a)}\frac{1}{1-x^{\chi(i,j)}q^{h_\lambda (i,j)}},
\end{multline*}
where 
$$
\chi(i,j)=\begin{cases} 1, &\text{ if } (i,j) \in ((n+c)^{n+a}),\\
0,& \text{ otherwise}.\end{cases}
$$
\end{thm}

\begin{proof}
Since both sides are power series in $x$ and $q$, is sufficient to show for $x=q^t$, where $t$ is an arbitrary integer. The idea is that we divide $\pi$ in two parts along the diagonal, compute the 
generating functions of the upper-right half and the lower-left half separately, and lastly combine them together.

The upper-right half is $\VV(n,a,b,m)$ (see Figure~\ref{fig:pi}).
We denote this upper-half by $\pi^{u}$ and the lower-left half by $\pi^{d}$. Note that $\pi^d = (\VV(n,c,d,m))'$. 

We consider $\pi^u$ first. 
 To utilize the generating function formula for the semistandard Young tableaux of shifted shapes 
with fixed diagonals given in Corollary \ref{cor:gfs}, we fill the top row of the skewed part $(\delta_{a +1})^\ast$ by $0$'s,  the second row by $1$'s, and so on, and the $a$-th row by $a-1$.
If we say we fixed the diagonal cells by $\nu_u=(\mu_1+n-1,\dots, \mu_{n-1},\mu_n ,a-1,\dots, 1,0)$, then we have 
\begin{multline*}
\sum_{\substack{T\in\SSYT(\pi^u)\\ \rdiag(T) = (\mu_1 +n-1, \dots, \mu_n)}} q^{|T| +t \cdot \tr (T)}
=q^{-\binom{a+1}{3}-t\binom{a}{2}}\sum_{\substack{T\in\SSYT(((b +1)^{n+a}+m\delta_{n+a})^\ast)\\ \rdiag(T) = \nu_u}} q^{|T|+t \cdot \tr (T)}\\
= \frac{q^{-\binom{a +1}{3}}\cdot q^{\binom{a}{2}+(t+1)|\mu+\delta_n|}}{\prod_{j=1}^{n+a}(q;q)_{b +m(n+a -j)}}
\overline{a}_{(b^{n+a})+m\delta_{n+a }}(q^{\mu_1 +n-1},\dots,q^{\mu_n },q^{a-1},\dots, q^0).
\end{multline*}
For the lower-left half, we consider the transpose of $\pi^d$, i.e., $(\pi^d )'= \VV(n,c,d,m)$. To satisfy the inequality condition of semistandard Young 
tableaux, after combined with the upper-right half, the fillings in this part should be row strict tableaux. 
Let us recall the generating function for the row strict tableaux with fixed diagonal :
\begin{equation}\label{eqn:RST}
 \sum_{\substack{T\in \RST((\delta_{n+1}+\rho)^\ast) \\ \rdiag (T)=\mu +\delta_n}} q^{|T|} =
\frac{q^{|\mu+\delta_n|+\nn(\rho')-\nn(\rho)+n|\rho|+\binom{n+1}{3}}}{\prod_{j=1}^n (q;q)_{\rho_j +n -j}} \overline{a}_{\rho+\delta_n }(q^{\mu+\delta_n}).
\end{equation}
To utilize the generating function for the row strict tableaux with fixed diagonal, we fill the skewed part $(\delta_{n+b+1})^\ast$ by 
$-b$ in the first column, $(-b+1)$'s in the second, and so on, and by $(-1)$'s in the last column. 
Then, in our setting, $\rho$ part in \eqref{eqn:RST} is $(d^{n+c})+(m-1)\delta_{n+c}$. We can compute 
\begin{align*}
& b(\rho') = (n+c)\binom{d}{2}+(m-1)^2 \binom{n+c+1}{3} + d (m-1)\binom{n+c }{2} -\binom{m}{2}\binom{n+c}{2},\\
& b(\rho) = d \binom{n+c}{2}+(m-1)\binom{n+c}{3},\\
& (c+n)|\rho| = (n+c)\left( d (n+c)+ (m-1)\binom{n+c }{2}\right).
\end{align*}
Let $p_\rho (c,d,n,m):= \nn(\rho')-\nn(\rho)+(n+c)|\rho|+\binom{n+c+1}{3}$. Then we have 
\begin{multline*}
\sum_{\substack{T\in\SSYT((\pi^d)')\\ \rdiag(T) = (\mu_1 +n-1,\dots, \mu_n)}} q^{|T|+t \cdot \tr (T)}
=q^{\binom{c +2}{3}+t\binom{c+1}{2}}\sum_{\substack{T\in\SSYT(((d+1) ^{n+c}+m\delta_{c+n})^\ast)\\ \rdiag(T) = (\mu+\delta_n, -1, \dots, -c)}} q^{|T|+t\cdot \tr (T)}\\
= \frac{q^{\binom{c +2}{3}-\binom{c +1}{2}+p_\rho(c,d,n,m)+(t+1)|\mu+\delta_n|}}{\prod_{j=1}^{n+c}(q;q)_{d+m(n+c -j)}}
\overline{a}_{(d^{n+c})+m\delta_{n+c}}( q^{\mu_1 +n-1},\dots, q^{\mu_n }, q^{-1}, \dots , q^{-c}).
\end{multline*}
Hence the trace generating function for the semistandard Young tableaux of shape $\pi$ would be 
\begin{align*} 
& \sum_{T\in\SSYT(\pi)} q^{|T|+t \cdot \tr (T)} \\
&= \sum_{\mu\in \Par_n}q^{-(t+1)|\mu+\delta_n|}\sum_{\substack{T\in\SSYT(\pi^u)\\ \rdiag(T) = \mu +\delta_n}} q^{|T| +t \cdot \tr (T)} \sum_{\substack{T\in\SSYT((\pi^d)')\\ \rdiag(T) = \mu +\delta_n}} q^{|T| +t \cdot \tr (T)}\\
&= \frac{q^{-\binom{a}{3}+\binom{c +1}{3}+p_\rho(c,d,n,m)}}{(1-q)^n \prod_{j=1}^{n+a}(q;q)_{b +m(n+a-j)}\prod_{j=1}^{n+c}(q;q)_{d +m(n+c-j)}}\\
&\qquad\times \int_{0\le x_1 \le \cdots \le x_n \le 1}
\prod_{i=1}^n x_i ^t \cdot \overline{a}_{(b^{n+a})+m\delta_{n+a}}(x_1,\dots, x_n, q^{a -1},\dots, q^{0})\\
&\qquad\qquad\qquad\quad\times \overline{a}_{(d^{n+c})+m\delta_{n+c}}(x_1,\dots, x_n, q^{-1}, \dots, q^{-c})d_q x_1\cdots d_q x_n.
\end{align*}
Note that we have already computed 
\begin{align*}
&\overline{a}_{(b^{n+a})+m\delta_{n+a }}(x_1,\dots, x_n, q^{a -1},\dots, q^{0})\\
&= q^{b\binom{a}{2}+m\binom{a}{3}+mn\binom{a}{2}}\prod_{i=1}^{a -1}(q^m ;q^m)_i \prod_{i=1}^n x_i ^{b }(q^{m(1-a)}x_i ^m;q^m)_{a}\cdot
\overline{\Delta}(x_1 ^m,\dots, x_n ^m),
\end{align*}
and similarly, we can compute 
\begin{align*}
& \overline{a}_{(d^{n+c})+m\delta_{n+c }}(x_1,\dots, x_n, q^{-1}, \dots, q^{-c})\\
&= q^{-d\binom{c+1}{2}-2m\binom{c +1}{3}-mn\binom{c +1}{2}}\prod_{i=1}^{c -1}(q^m ;q^m)_i \prod_{i=1}^n x_i ^{d }(q^{m}x_i ^m;q^m)_{c}\cdot
\overline{\Delta}(x_1 ^m,\dots, x_n ^m).
\end{align*}
Combining all these gives 
\begin{align*}
& \sum_{T\in\SSYT(\pi)} q^{|T|+t\cdot \tr (T)} \\
&= q^{(b+mn)\binom{a}{2}-(d+mn)\binom{c+1}{2}+(m-1)\binom{a}{3}+(1-2m)\binom{c+1}{3}+p_\rho(c,d,n,m)}\\
&\quad\times \frac{\PPhi_{q^m}(a) \PPhi_{q^m}(c)}{(1-q)^n \prod_{j=0}^{n+a-1}(q;q)_{b +mj}\prod_{j=0}^{n+c-1}(q;q)_{d +m j}}\\
&\quad\times \int_{0\le x_1 \le \cdots \le x_n\le 1}\prod_{i=1}^n x_i ^{t+b+d } (q^{m(1-a)}x_i ^m ;q^m)_{a+c}\cdot \overline{\Delta} (x_1 ^m,\dots, x_n ^m) ^2
d_q x_1 \cdots d_q x_n.
\end{align*}
Here we used the notation $\PPhi_{q^m}(n)=\prod_{j=1}^{n-1}(q^m ;q^m)_j$. 
By making a change of variables $x_i ^m \mapsto x_i$ and letting $p=q^m$ in the $q$-integral, we get
\begin{align*}
& \sum_{T\in\SSYT(\pi)} q^{|T|+t\cdot \tr (T)} \\
&= q^{(b+mn)\binom{a}{2}-(d+mn)\binom{c+1}{2}+(m-1)\binom{a}{3}+(1-2m)\binom{c+1}{3}+p_\rho(c,d,n,m)}\\\
&\quad\times \frac{\PPhi_{q^m}(a) \PPhi_{q^m}(c)}{(1-q^m)^n \prod_{j=0}^{n+a-1}(q;q)_{b +mj }\prod_{j=0}^{n+c-1}(q;q)_{d +m j}}\\
&\quad\times \underbrace{\int_{0\le x_1 \le \cdots \le x_n\le 1}\prod_{i=1}^n x_i ^{\frac{t+b+d+1-m}{m}} (p^{1-a}x_i  ;p)_{a+c}\cdot \overline{\Delta}(x_1 ,\dots, x_n )^2 
d_p x_1 \cdots d_p x_n.}_{(\ast)}
\end{align*}
The last integral $(\ast)$ can be calculated by using \eqref{eqn:qSelbergint} with $\mathsf{a}=0$. When $\mathsf{a}=0$, 
\eqref{eqn:qSelbergint} becomes 
\begin{align*}
&\int_{0\le X \le \mathsf{b}} \Delta(X)^2 \prod_{i=1}^n x_i ^{\alpha -1}\left(\frac{q x_i}{\mathsf{b}} \right)_{\beta -1}d_q x_1 \cdots d_q x_n\\
&= q^{\binom{n}{3}-n\binom{\alpha}{2}}\prod_{i=1}^n \frac{\Gamma_q (\alpha-1+i)\Gamma_q (\beta-1+i)\Gamma_q (i)\cdot q^{\binom{\alpha-1+i}{2}}\cdot \mathsf{b}^{\alpha+2i-2}}{\Gamma_q (\alpha-1+\beta-1+n+i)}\\
&= q^{(\alpha-1)\binom{n}{2}+\frac{1}{6}n(n-1)(2n-1)}\cdot \mathsf{b}^{n\alpha+n(n-1)}(1-q)^n
\prod_{i=1}^n \frac{(q;q)_{\beta+i-2}(q;q)_{i-1}}{(q^{\alpha-1+i};q)_{n+\beta-1}}.
\end{align*}
If we let 
\begin{align*}
\alpha&= \frac{t+b+d +1-m}{m}+1,\\
\beta&= a+c +1,\\
\mathsf{b}&= p^{a}=q^{a\cdot m }
\end{align*}
to evaluate $(\ast)$, then we obtain 
\begin{align*}
&\int_{0\le x_1 \le \cdots \le x_n\le 1}\prod_{i=1}^n x_i ^{\frac{t+b+d+1-m}{m}} (p^{1-a}x_i  ;p)_{a+c}\cdot \overline{\Delta}(x_1 ,\dots, x_n ) ^2 
d_p x_1 \cdots d_p x_n\\
&=q^{(t+b+d+1)\left(a n +\binom{n}{2} \right)+2m\left( \binom{n}{3}+a  \binom{n}{2}\right)}
\dfrac{(1-q^m )^n \PPhi_{q^m}(n)\PPhi_{q^m}(a+c+n)}{\PPhi_{q^m}(a+c)\prod_{i=0}^{n-1}(q^{t+b+d+1-mi};q^m)_{a+c+n}}.
\end{align*}
Applying the above result of integration gives 
\begin{multline*}
\sum_{T\in\SSYT(\pi)} q^{|T|+t\cdot \tr (T)}  \\
=q^{(\ast\ast)}\dfrac{\PPhi_{q^m}(a)\PPhi_{q^m}(c)\PPhi_{q^m}(n)\PPhi_{q^m}(a+c+n)}{\PPhi_{q^m}(a+c)\prod_{j=0}^{a+n-1}(q;q)_{b+mj}\prod_{j=0}^{c+n-1}(q;q)_{d+mj}\prod_{i=0}^{n-1}(q^{t+b+d+1-mi};q^m)_{a+c+n}},
\end{multline*}
where 
\begin{align*}
(\ast\ast)= & t\left( a n  +\binom{n}{2}\right)\\
+& (b +mn)\binom{a}{2}-(d +mn)\binom{c+1}{2}+(m-1)\binom{a}{3}+(1-2m)\binom{c+1}{3}\\
+& (b+d+1)\left( a n  +\binom{n}{2}\right)+2m \left(\binom{n}{3}+a\binom{n}{2} \right) +p_\rho (c,d,n, m).
\end{align*}
Note that 
$$
(\ast\ast)- t\left( a n  +\binom{n}{2}\right)
= \sum_{(i,j)\in \lambda/(c)^{a}} (\lambda_j '-i),
$$
where  $\lambda=((n+b+c)^{n+a})+((m-1)\delta_{a+n}\cup \theta ')$, $\theta=(d^{n+c})+(m-1)\delta_{n+c}$.
Also note that 
\begin{align*}
&\frac{\PPhi_{q^m}(a)\PPhi_{q^m}(c)\PPhi_{q^m}(n)\PPhi_{q^m}(n+a+c)}{\PPhi_{q^m}(a+c)\prod_{j=0}^{n+a-1}(q;q)_{b+mj}\prod_{j=0}^{n+c-1}(q;q)_{d+mj}\prod_{i=0}^{n-1}(q^{t+b+d+1-mi};q^m)_{n+a+c}}\\
&= \prod_{i=1}^n \prod_{j=1}^{a}\prod_{k=1}^{c}\frac{1-q^{m(i+j+k-1)}}{1-q^{m(i+j+k-2)}}\cdot 
\prod_{(i,j)\in \lambda/(0^n, c^a)}\frac{1}{1-q^{t\cdot \chi(i,j)+h_\lambda (i,j)}},
\end{align*}
where 
$$\chi(i,j)=\begin{cases}
1,& \text{ if } (i,j)\in ((n+c)^{n+a}),\\
0, & \text{ otherwise.}
\end{cases}$$
Lastly, we replace $q^t$ by $x$. 
\end{proof}

\section*{Acknowledgement}
The authors are grateful to Igor Pak and Ole Warnaar for their helpful comments.

\end{document}